\documentclass[a4paper]{article}
\usepackage{enumitem}
\usepackage{xcolor}
\usepackage{upgreek}

\usepackage{amsmath,amssymb,amsthm,mathtools}
\usepackage{tikz,graphicx}
\usepackage{tkz-euclide}
\usepackage{pict2e}
\usepackage{todonotes}
\theoremstyle{plain}

\newtheorem*{theorem*}{Theorem}
\newtheorem{theorem}{Theorem}[section] 

\newtheorem{proposition}[theorem]{Proposition}

\newtheorem{remark}[theorem]{Remark}
\theoremstyle{definition}
\newtheorem{assumption}[theorem]{Assumption}
\newtheorem{example}[theorem]{Example}
\newtheorem{definition}[theorem]{Definition}

\usepackage{tikz}
\usepackage{pict2e}
\usetikzlibrary{arrows,calc,chains, positioning, shapes.geometric,shapes.symbols,decorations.markings,arrows.meta}
\usetikzlibrary{patterns,angles,quotes}
\usetikzlibrary{decorations.markings}
\tikzset{middlearrow/.style={
			decoration={markings,
				mark= at position 0.6 with {\arrow{#1}} ,
			},
			postaction={decorate}
		}
	}
\tikzset{->-/.style={decoration={
				markings,
				mark=at position #1 with {\arrow{latex}}},postaction={decorate}}}
	
	\tikzset{-<-/.style={decoration={
				markings,
				mark=at position #1 with {\arrowreversed{latex}}},postaction={decorate}}}

\newcommand{\ds}{\displaystyle}

\numberwithin{equation}{section}

\def\bigO{{\cal O}}

\newcommand{\z}{\mathrm{z}}
\newcommand{\w}{\mathrm{w}}

\newcommand{\e}{\mathfrak{e}}

\tikzset{
	master/.style={
		execute at end picture={
			\coordinate (lower right) at (current bounding box.south east);
			\coordinate (upper left) at (current bounding box.north west);
		}
	},
	slave/.style={
		execute at end picture={
			\pgfresetboundingbox
			\path (upper left) rectangle (lower right);
		}
	}
}

\tikzset{middlearrow/.style={
		decoration={markings,
			mark= at position 0.6 with {\arrow{#1}} ,
		},
		postaction={decorate}
	}
}

\let\oldbibliography\thebibliography
\renewcommand{\thebibliography}[1]{\oldbibliography{#1}
\setlength{\itemsep}{-0.5pt}}

\usepackage[T1]{fontenc}
\usepackage[UKenglish,french,spanish,italian,USenglish]{babel}
\usepackage[useregional]{datetime2}
\newcommand{\mydate}{\DTMdisplaydate{2021}{11}{24}{-1}}
\date{\selectlanguage{USenglish}
\mydate} 

\newcommand{\im}{\text{\upshape Im\,}}

\DeclareMathOperator{\diag}{diag}

\def\XXint#1#2#3{{\setbox0=\hbox{$#1{#2#3}{\oint}$ }
\vcenter{\hbox{$#2#3$ }}\kern-.59\wd0}}

\tikzset{
    master/.style={
        execute at end picture={
            \coordinate (lower right) at (current bounding box.south east);
            \coordinate (upper left) at (current bounding box.north west);
        }
    },
    slave/.style={
        execute at end picture={
            \pgfresetboundingbox
            \path (upper left) rectangle (lower right);
        }
    }
}

\usepackage[colorlinks=true]{hyperref}
\hypersetup{urlcolor=blue, citecolor=red, linkcolor=blue}

\begin{document}

\title{Matrix orthogonality in the plane versus \\ scalar orthogonality in a Riemann surface}
\author{Christophe Charlier \\
\small \textit{Department of Mathematics, KTH Royal Institute of Technology,} \\ \small \textit{100 44 Stockholm, Sweden}.}

\maketitle

\begin{abstract}
We consider a non-Hermitian matrix orthogonality on a contour in the complex plane. Given a diagonalizable and rational matrix valued weight, we show that the Christoffel--Darboux (CD) kernel, which is built in terms of matrix orthogonal polynomials, is equivalent to a scalar valued reproducing kernel of meromorphic functions in a Riemann surface. If this Riemann surface has genus $0$, then the matrix valued CD kernel is equivalent to a scalar reproducing kernel of polynomials in the plane. Interestingly, this scalar reproducing kernel is not necessarily a scalar CD kernel. As an application of our result, we show that the correlation kernel of certain doubly periodic lozenge tiling models admits a double contour integral representation involving only a scalar CD kernel. This simplifies a formula of Duits and Kuijlaars.


\end{abstract}



\section{Introduction and statement of results}\label{section: introduction}
$P$ is an $r \times r$ matrix polynomial of degree $N$ if it can be written in the form 
\begin{align*}
P(z) = \sum_{k=0}^{N}C_{k}z^{k} \quad \mbox{ for some } C_{0},\ldots,C_{N} \in \mathbb{C}^{r \times r} \mbox{ with } C_{N} \neq 0_{r},
\end{align*}
where $0_{r}$ denotes the $r \times r$ zero matrix. $P$ is monic if $C_{N}$ is equal to the identity matrix $I_{r}$.

\medskip Consider the following non-Hermitian bilinear pairing between $r \times r$ matrix polynomials
\begin{align}\label{pairing}
\langle P,Q \rangle = \int_{\gamma} P(z)W(z)Q(z)dz,
\end{align}
where $\gamma \subset \mathbb{C}$ is a union of finitely many piecewise smooth, oriented curves and $W$ is a continuous $r \times r$ matrix weight. We say that $P$ and $Q$ are orthogonal if $\langle P,Q \rangle = 0$. Because the matrix product does not commute, in general one has $\langle P,Q \rangle \neq \langle Q,P \rangle$, and therefore the pairing \eqref{pairing} gives rise to two families of matrix orthogonal polynomials (MOPs): the left MOPs and the right MOPs. For example, the two monic MOPs of degree $j$, denoted by $P_{j}^{\mathrm{L}}(z) = z^{j}I_{r}+...$ and 
$P_{j}^{\mathrm{R}}(z) = z^{j}I_{r}+...$, are defined such that
\begin{align*}
& \langle P_{j}^{\mathrm{L}},z^{k}I_{r} \rangle = 0_{r}, & & \langle z^{k}I_{r}, P_{j}^{\mathrm{R}} \rangle =0_{r}, & & k=0,\ldots,j-1.
\end{align*}
A particular feature of the matrix orthogonality considered in this work is that the bilinear pairing \eqref{pairing} is \textit{not} an inner product, and therefore in our setting there is no guarantee that the MOPs $P_{j}^{\mathrm{L}}$ and $P_{j}^{\mathrm{R}}$ exist and are unique for every $j$ (except in the special case where $\gamma \subset \mathbb{R}$ and $W(z)$ is symmetric positive definite for every $z \in \gamma$), see also Section \ref{subsubsection: from matrix to scalar OPs} below. 


\medskip We will restrict our attention to \textit{rational} matrix weights, that is, each entry of $W$ is a rational function with no pole on $\gamma$. This type of matrix orthogonality arises in the theory of tiling models \cite{DK}, see also Section \ref{section: applications}. Our main result will be valid under the following assumption on the weight $W$.
\begin{assumption}\label{ass: weight diag}
The rational $r \times r$ matrix valued function $z \mapsto W(z)$ has no pole on $\gamma$ and is diagonalizable for all but finitely many $z \in \mathbb{C}$.
\end{assumption}
Since $W$ is rational, its eigenvalues $\lambda_{1},\ldots,\lambda_{r}$ are (branches of) meromorphic functions on $\mathbb{C}$. Assumption \ref{ass: weight diag} implies that for all but finitely many $z \in \mathbb{C}$, there exists an invertible $r \times r$ matrix $E(z)$ such that
\begin{equation}\label{eigenvalue eigenvector decomposition of A}
W(z) = E(z) \Lambda(z) E(z)^{-1},
\end{equation}
where $\Lambda(z) = \diag(\lambda_{1}(z),\ldots,\lambda_{r}(z))$. Note that it does not matter for us if $W$ fails to be diagonalizable for finitely many points on the contour $\gamma$ itself. Assumption \ref{ass: weight diag} essentially rules out the weights with a Jordan block structure. For example, the weight
\begin{align*}
\begin{pmatrix}
1 & z \\ 0 & 1
\end{pmatrix}
\end{align*}
is only diagonalizable at $z=0$, and therefore does not satisfy Assumption \ref{ass: weight diag}. Our main result is stated in Theorem \ref{thm: reproducing kernel Riemann surface} and is described in terms of the Christoffel-Darboux (CD) kernel. It can roughly be summarized as \textit{if $W$ satisfies Assumption \ref{ass: weight diag}, the non-Hermitian matrix orthogonality induced by \eqref{pairing} is equivalent to a scalar orthogonality in a Riemann surface}. If this Riemann surface has genus $0$, it can be mapped to the plane and the matrix orthogonality is equivalent to a scalar orthogonality in the plane. For instance, we show in Example \ref{example: scalar 2} below that the orthogonality on the unit circle associated to the $r \times r$ matrix weight
\begin{align}\label{W example rxr}
W(z) = z^{-R} \begin{pmatrix}
1 & 1 & 0 & 0 & \cdots & 0 & 0 \\
0 & 1 & 1 & 0 & \cdots & 0 & 0 \\
\vdots & \vdots & \vdots & \vdots & \ddots & \vdots & \vdots \\
0 & 0 & 0 & 0 & \cdots & 1 & 1 \\
z & 0 & 0 & 0 & \cdots & 0 & 1
\end{pmatrix}^{L}, \qquad L,R \in \mathbb{N}_{>0},
\end{align}
is equivalent to the scalar orthogonality, also on the unit circle, associated to the well-studied \cite{MFO} Jacobi weight $\zeta^{-rR}(1+\zeta)^{L}$ with non-standard parameters. (For clarity, we mention that the matrix in \eqref{W example rxr} only contains ones on the main diagonal and on the first upper diagonal). We provide some applications of our results to the theory of tiling models in Section \ref{section: applications}. 

\medskip This work is inspired from \cite{Charlier}, in which it was established that the matrix orthogonality on the unit circle associated to
\begin{align}\label{def of W}
W(z) = \frac{1}{z^{n}}\begin{pmatrix}
\alpha^{2} + z & 1+\alpha \\
(1+\alpha^{3})z & 1+\alpha^{2}z
\end{pmatrix}^{n}, \qquad \alpha \in (0,1), \; n \in \mathbb{N}_{>0}
\end{align}
is equivalent to the scalar orthogonality associated to
\begin{align*}
\mathcal{W}(\zeta) = \bigg( \frac{(\zeta-\alpha c)(\zeta - \alpha c^{-1})}{\zeta (\zeta-c)(\zeta-c^{-1})} \bigg)^{n}, \qquad \mbox{ where } \qquad c = \sqrt{\frac{\alpha}{1-\alpha+\alpha^{2}}},
\end{align*}
on a suitable contour which we do not describe here. Several steps in the proof of \cite{Charlier} rely on the exact expression \eqref{def of W} of $W$. Here, we generalize these ideas to handle any $r \times r$ matrix weight $W$ satisfying Assumption \ref{ass: weight diag}. 

\begin{remark}
In the context of matrix orthogonality on the real line, it was already observed by several authors, see e.g. \cite{DurGr, CanMoVel}, that if the matrix weight is diagonalizable with \textbf{constant} matrices, then the matrix orthogonality is nothing really different from a scalar orthogonality on the real line. The major difference with our situation is that $z \mapsto E(z)$ is obviously not necessarily constant.
\end{remark}

\paragraph{Earlier works.} The study of MOPs has been initiated by \cite{KreinCircle, KreinLine}, motivated by a moment problem arising in operator theory. MOPs have then been studied sporadically, until a resurgence in the 1980's. They have found applications in scattering theory \cite{Ger,AN}, matrix valued spherical functions \cite{Koo,GPT,KMR I,KMR II,KPR, PR,AKR}, system theory \cite{Fuhrmann}, Gaussian quadrature for matrix functions \cite{SVA}, the analysis of sequences of polynomials satisfying higher order recurrence relations \cite{Dur, DVA}, integrable systems \cite{Mir,CafassoIgl CMP, CafassoIgl, AM,  IKR}, Toda lattices \cite{AFAGAMM, DER}, among others. For a survey on MOPs up to 2008, we refer to \cite{DPS}. The standard matrix orthogonality that one often encounters in the literature is associated with a \textit{Hermitian} bilinear pairing
\begin{align}\label{second pairing lol}
( P,Q ) = \int_{\gamma} P(z)W(z)Q(z)^{\dag}dz,
\end{align}
where $^\dag$ denotes the transpose conjugate operation, and where 
\begin{itemize}
\item \vspace{-0.1cm} either $\gamma \subset \mathbb{R}$, and $W(z)$ is real valued, symmetric and positive definite for every $z \in \gamma$,
\item \vspace{-0.2cm} or $\gamma = \{z \in \mathbb{C}: |z|=1\}$ is oriented positively, and $W(z)\frac{dz}{|dz|} = izW(z)$ is Hermitian and positive definite for every $z \in \gamma$.
\end{itemize}
\vspace{-0.1cm}In each of these two cases, the positive definiteness property ensures that the pairing \eqref{second pairing lol} is an inner product, and therefore the existence of the MOPs is guaranteed, see also \cite{SVA1996} for more details. In this work, we deviate from these standard set-ups in several aspects: we consider a non-Hermitian matrix orthogonality, $\gamma$ is a general contour in the complex plane and $W$ is not necessarily symmetric or Hermitian (let alone positive definite). This type of orthogonality appears in the study of certain tiling models and was first considered by \cite{DK}.


\medskip We now introduce the necessary material to state our results.
\paragraph{CD kernel.} Given a contour $\gamma$, a matrix weight $W$ and $N \in \mathbb{N}_{>0}$, the associated CD kernel $\mathcal{R}_{N}^{W}(w,z)$ is defined as the unique bivariate $r \times r$ matrix polynomial of degree $\leq N-1$ in both $w$ and $z$ that satisfies either
\begin{align}
& \int_{\gamma} P(w)W(w) \mathcal{R}_{N}^{W}(w,z)dw = P(z), & & \mbox{for every } P \in \mathcal{P}_{N-1}^{r \times r} \mbox{ and } z \in \mathbb{C}, \label{reproducing kernel} \\
\mbox{or } \quad & \int_{\gamma} \mathcal{R}_{N}^{W}(w,z)W(z)P(z) dz = P(w), & & \mbox{for every } P \in \mathcal{P}_{N-1}^{r \times r} \mbox{ and } w \in \mathbb{C}, \label{reproducing kernel 2}
\end{align}
where for $j,k \in \mathbb{N}_{>0}$, $\mathcal{P}_{N-1}^{j \times k} = \{\sum_{\ell = 0}^{N-1}C_{\ell}z^{\ell}:C_{0},\ldots,C_{N-1} \in \mathbb{C}^{j \times k}\}$ is the vector space of all $j \times k$ matrix polynomials of degree $\leq N-1$. Because of \eqref{reproducing kernel}--\eqref{reproducing kernel 2}, $\mathcal{R}_{N}^{W}$ is also called the \textit{reproducing kernel} for $\mathcal{P}_{N-1}^{r \times r}$, and we refer to \cite[Proposition 2.4]{Delvaux} (see also \cite[Lemma 4.6]{DK}) for a proof that $\mathcal{R}_{N}^{W}$ is indeed unique. There also exists an explicit expression for $\mathcal{R}_{N}^{W}$ in terms of the left and right MOPs, but since this expression is not needed to state our results, we defer it to Section \ref{subsubsection: from matrix to scalar OPs}. In our setting, because we deal with a non-Hermitian matrix orthogonality, there is in general no guarantee of existence for $\mathcal{R}_{N}^{W}$; however we mention that in certain situations one can prove the existence of $\mathcal{R}_{N}^{W}$ using the particular structure of $W$, see \cite[Lemma 4.8]{DK}. If $\gamma$ is unbounded, then a necessary condition for the existence of $\mathcal{R}_{N}^{W}$ is
\begin{align*}
W(z) = \bigO(z^{-2N}), \qquad z \to \infty.
\end{align*}
This condition ensures the convergence of the integrals in \eqref{reproducing kernel}--\eqref{reproducing kernel 2}.

\medskip Since $W$ is rational, we can write $\det(W(z) - \lambda I_{r} ) = P_{W}(z,\lambda)/Q_{W}(z)$ for some polynomials $P_{W}$ and $Q_{W}$. Let us consider the zero set of $P_{W}$, namely
\begin{align}\label{Riemann surface M general case}
\{(z,\lambda) \in \mathbb{C}^{2} : P_{W}(z,\lambda)=0\}.
\end{align}
It is well-known, see e.g. \cite[Example 6 of Section 4.2 and Chapter 5]{Schlag}, that any zero set of a polynomial in two variables can be completed to an algebraic curve (=compact Riemann surface). Since $P_{W}$ is of degree $r$ in the variable $\lambda$, the Riemann surface $\mathcal{M}$ associated to \eqref{Riemann surface M general case} can be represented as an $r$-sheeted covering of $\widehat{\mathbb{C}}:= \mathbb{C}\cup \{\infty\}$ (and $\mathcal{M}$ is connected if and only if $P_{W}$ is irreducible). We will use the notation $\z, \w$ to denote points of $\mathcal{M}$, and $z,w$ for points of $\mathbb{C}$. If $\z$ and $z$ appear in the same equation, then $z$ denotes the projection of $\z$ on $\widehat{\mathbb{C}}$ (and similarly for $\w$ and $w$). We choose the numbering of the sheets such that the function
\begin{align}
& \z \mapsto \lambda(\z) := \lambda_{k}(z), & & \mbox{if $\z $ is on the $k$-th sheet of $\mathcal{M}$}, \label{def of lambda RS} 
\end{align}
is meromorphic on $\mathcal{M}$ (note that this numbering is not unique). Assumption \ref{ass: weight diag} implies, see Appendix \ref{section: appendix eigenvector} for the details, that we can (and do) choose the matrix of eigenvectors $E(z)$ such that the functions
\begin{align}
& \z \mapsto \mathfrak{e}(\z) := E(z)e_{k}, & & \mbox{if $\z$ is on the $k$-th sheet of $\mathcal{M}$}, \label{def of efrak RS} \\
& \z \mapsto \widehat{\e}(\z) := e_{k}^{T}E(z)^{-1}, & & \mbox{if $\z$ is on the $k$-th sheet of $\mathcal{M}$}, \label{def of efrak inv RS}
\end{align}
are also entrywise meromorphic on $\mathcal{M}$, where $e_{k}$ is the $k$-th column of the identity matrix and $^T$ denotes the transpose operation. 

\medskip The Riemann surface $\mathcal{M}$ associated with the zero set \eqref{Riemann surface M general case} is just one example of a Riemann surface for which one can define meromorphic functions $\lambda$, $\e$ and $\widehat{\e}$ as in \eqref{def of lambda RS}--\eqref{def of efrak inv RS}. However, in certain cases it is more convenient to work with a slightly different $\mathcal{M}$. Assume for example that the weight $W$ is of the form $W(z) = A(z)^{L}$ for a certain rational matrix $A$ and a certain $L \in \mathbb{N}_{>0}$, and assume that
\begin{align*}
A(z) = E(z)\widetilde{\Lambda}(z)E(z)^{-1}, \qquad \widetilde{\Lambda}(z) = \diag(\widetilde{\lambda}_{1}(z),\ldots,\widetilde{\lambda}_{r}(z)), \qquad \mbox{ for all but finitely many }z \in \mathbb{C}.
\end{align*} 
The eigenvalues of $W$ are obviously $\lambda_{k}(z)=\widetilde{\lambda}_{k}(z)^{L}$, $k=1,\ldots,r$. Therefore, in this case, instead of using the Riemann surface associated to \eqref{Riemann surface M general case}, $\mathcal{M}$ can be defined as the (simpler) algebraic curve constructed from the zero set $\{(z,\widetilde{\lambda}) \in \mathbb{C}^{2} : P_{A}(z,\widetilde{\lambda})=0\}$. 

\medskip From now, $\mathcal{M}$ denotes an arbitrary $r$-sheeted Riemann surface such that the functions $\lambda$, $\e$ and $\widehat{\e}$ defined in \eqref{def of lambda RS}--\eqref{def of efrak inv RS} are meromorphic. Assumption \ref{ass: weight diag} implies the existence of such a $\mathcal{M}$, but the exact choice of the zero set from which $\mathcal{M}$ is constructed does not matter for our results.

\begin{remark}
Since $\mathfrak{e}$ is of size $r \times 1$ and $\widehat{\e}$ is of size $1 \times r$, there are not the inverse of each other in the usual matrix sense. However, from the relations $E(z)E(z)^{-1} = I_{r} = E(z)^{-1}E(z)$, we deduce that they satisfy\footnote{Strictly speaking, \eqref{inverse e 1}-\eqref{inverse e 2} hold for each $z \in \mathbb{C}$ that is not a pole of $E$ or $E^{-1}$, but by continuity they hold for all $z \in \mathbb{C}$.}
\begin{align}
& \widehat{\e}(z^{(j)})\mathfrak{e}(z^{(k)}) = \delta_{j,k}, & & \mbox{for all } z \in \mathbb{C} \mbox{ and } 1 \leq j,k \leq r, \label{inverse e 1} \\
& \sum_{j=1}^{r} \mathfrak{e}(z^{(j)})\widehat{\e}(z^{(j)}) = I_{r}, & & \mbox{for all } z \in \mathbb{C}, \label{inverse e 2}
\end{align}
where for a given $z \in \mathbb{C}$, $z^{(k)}$ denotes the point on the $k$-th sheet of $\mathcal{M}$ whose projection on $\mathbb{C}$ is $z$.
\end{remark}
Let us illustrate with an example how to compute in practice the functions $\lambda$, $\e$ and $\widehat{\e}$.
\begin{example}\label{example: first}
Consider the weight
\begin{align*}
W(z) = \begin{pmatrix}
1 & 1 \\ z^{k} & 1
\end{pmatrix}, \qquad k \in \mathbb{Z}.
\end{align*}
If $k \hspace{-0.15cm} \mod 2 = 1$, then we define $z^{\frac{k}{2}} = |z|^{\frac{k}{2}}e^{\frac{i k}{2}\arg z}$, where the principal branch is chosen for $\arg z$, i.e. $\arg z \in (-\pi,\pi)$. For all $z \in \mathbb{C}\setminus \{0\}$, we can write $W(z) = E(z) \Lambda(z) E(z)^{-1}$ with
\begin{align}\label{decomposition in example 1}
E(z)= \begin{pmatrix}
1 & 1 \\ z^{\frac{k}{2}} & -z^{\frac{k}{2}}
\end{pmatrix}, \qquad \Lambda(z) = \begin{pmatrix}
1+z^{\frac{k}{2}} & 0 \\
0 & 1-z^{\frac{k}{2}}
\end{pmatrix}, \qquad E(z)^{-1}= \frac{1}{2}\begin{pmatrix}
1 & z^{-\frac{k}{2}} \\
1 & -z^{-\frac{k}{2}}
\end{pmatrix}.
\end{align}
If $z<0$ and $k \hspace{-0.15cm} \mod 2 = 1$, then one needs to substitute $z^{\frac{k}{2}}$ by $|z|^{\frac{k}{2}}e^{\frac{\pi i k}{2}}$ everywhere in \eqref{decomposition in example 1} (or by $|z|^{\frac{k}{2}}e^{-\frac{\pi i k}{2}}$ everywhere). We let $\mathcal{M}$ be the Riemann surface associated to the zero set $\{(z,\eta)\in \mathbb{C}^{2}: \eta^{2} = z^{k}\}$. We choose the numbering of the sheets such that $\eta=z^{\frac{k}{2}}$ on the first sheet and $\eta=-z^{\frac{k}{2}}$ on the second sheet (if $k \hspace{-0.15cm} \mod 2 = 0$, $\mathcal{M}$ is simply the disjoint union of two copies of $\widehat{\mathbb{C}}$). The functions $\lambda$, $\e$ and $\widehat{\e}$, defined by \eqref{def of lambda RS} and \eqref{def of efrak RS}--\eqref{def of efrak inv RS}, are explicitly given by 
\begin{align*}
\lambda((z,\eta)) = 1+\eta, \qquad  \e((z,\eta)) = \begin{pmatrix}
1 & \eta
\end{pmatrix}^{T}, \qquad \widehat{\e}((z,\eta)) = \frac{1}{2}\begin{pmatrix}
1 & \eta^{-1}
\end{pmatrix},
\end{align*}
where $\z = (z,\eta)$ denotes a point of $\mathcal{M}$. These functions are meromorphic on $\mathcal{M}$, as required. Note that the use of $(z,\eta)$ for a point of $\mathcal{M}$ is a slight abuse of notation, because e.g. if $k=2$, $\mathcal{M}$ is the disjoint union of two copies of $\widehat{\mathbb{C}}$, and this notation does not distinguish the points $0^{(1)}$ and $0^{(2)}$, which are both denoted $(0,0)$. We will use again this notation several times for convenience, but if this can lead to confusion we will clarify it. 
\end{example}
\paragraph{Main results.} Let $\mathcal{M}_{*}$ be the set $\mathcal{M}$ with all points at infinity removed, and let $\mathcal{Q}$ and $\widehat{\mathcal{Q}}$ be the finite sets of all poles of $\e$ and $\widehat{\e}$, respectively.\footnote{$\z$ is a pole of $\e$ (resp. of $\widehat{\e}$) if $\z$ is a pole for at least one entry of $\e$ (resp. of $\widehat{\e}$).} For $\z \in \mathcal{M}_{*}\setminus \mathcal{Q}$ and $\w \in \mathcal{M}_{*}\setminus \widehat{\mathcal{Q}}$, we define
\begin{align}\label{def of RM}
R_{N}^{\lambda}(\w,\z) = \widehat{\e}(\w)\mathcal{R}_{N}^{W}(w,z)\e(\z),
\end{align}
where we recall that $z$ and $w$ denote the projections on the complex plane of $\z$ and $\w$, respectively. Note that $R_{N}^{\lambda}$ is scalar valued, but is equivalent to $\mathcal{R}_{N}^{W}$ in the sense that we can completely recover $\mathcal{R}_{N}^{W}$ from $R_{N}^{\lambda}$ by
\vspace{-0.2cm}\begin{align}\label{equivalence in equation}
\big[ R_{N}^{\lambda}(w^{(j)},z^{(k)}) \big]_{j,k=1}^{r} = E(w)^{-1}\mathcal{R}_{N}^{W}(w,z)E(z).
\end{align}
Let $\gamma_{\mathcal{M}} = \cup_{j=1}^{r} \gamma^{(j)}$ be the contour on $\mathcal{M}$ that consists of $r$ copies of $\gamma$, one on each sheet, and let $L_{N}$ and $\widehat{L}_{N}$ be the vector spaces of scalar meromorphic functions on $\mathcal{M}$ given by
\begin{align}\label{def of LN and LN star}
& L_{N} = \{\z \mapsto P(z)\e(\z): P\in \mathcal{P}_{N-1}^{1 \times r} \}, & & \widehat{L}_{N} = \{\z \mapsto \widehat{\e}(\z)P(z): P\in \mathcal{P}_{N-1}^{r \times 1} \}.
\end{align}
Our first main result states that $R_{N}^{\lambda}$ satisfies some reproducing properties for $L_{N}$ and $\widehat{L}_{N}$ on the contour $\gamma_{\mathcal{M}}$.
\begin{theorem}\label{thm: reproducing kernel Riemann surface}
Let $\gamma \subset \mathbb{C}$ be a finite union of piecewise smooth, oriented curves, let $W$ be a rational $r \times r$ matrix weight, and let $N \in \mathbb{N}_{>0}$. Suppose that $W$ satisfies Assumption \ref{ass: weight diag} and that $\mathcal{R}_{N}^{W}$ exists. Let $\mathcal{M}$ be an $r$-sheeted Riemann surface such that the functions $\lambda$, $\e$ and $\widehat{\e}$ defined in \eqref{def of lambda RS}--\eqref{def of efrak inv RS} are meromorphic, and define $R_{N}^{\lambda}$ as in \eqref{def of RM}. Then $R_{N}^{\lambda}$ exists, $\dim L_{N} = \dim \widehat{L}_{N}=rN$, and we have the following:
\begin{itemize}
\item[(a)] $\z \mapsto R_{N}^{\lambda}(\w,\z) \in L_{N}$ for every $\w \in \mathcal{M}_{*} \setminus \widehat{\mathcal{Q}}$, 
\item[(b)] $\w \mapsto R_{N}^{\lambda}(\w,\z) \in \widehat{L}_{N}$ for every $\z \in \mathcal{M}_{*} \setminus \mathcal{Q}$, 
\item[(c)] $R_{N}^{\lambda}$ satisfies the following reproducing property for $L_{N}$:
\begin{align}\label{reproducing property M}
& \int_{\gamma_{\mathcal{M}}} f(\w) \lambda(\w) R_{N}^{\lambda}(\w,\z) dw = f(\z), & & \mbox{for every } f \in L_{N} \mbox{ and } \z \in \mathcal{M}_{*} \setminus \mathcal{Q}.
\end{align}
\item[(d)] $R_{N}^{\lambda}$ satisfies the following reproducing property for $\widehat{L}_{N}$:
\begin{align}\label{reproducing property M 2}
& \int_{\gamma_{\mathcal{M}}}  R_{N}^{\lambda}(\w,\z) \lambda(\z) f(\z) dz = f(\w), & & \mbox{for every } f \in \widehat{L}_{N} \mbox{ and } \w \in \mathcal{M}_{*} \setminus \widehat{\mathcal{Q}}.
\end{align}
\end{itemize}
\end{theorem}
\begin{proof}
See Section \ref{proof kernel M}.
\end{proof}
$L_{N}$ and $\widehat{L}_{N}$ are contained in certain spaces of meromorphic functions with prescribed zeros and allowed poles. To describe this relation, we briefly introduce some notation and definitions. 
\begin{definition}(divisors)
A divisor $D$ is a formal sum of the form $D = \sum_{j=1}^{n}\ell_{j}\z_{j}$, where $n \in \mathbb{N}_{>0}$, $\ell_{j} \in \mathbb{Z}$ and $\z_{j} \in \mathcal{M}$. We will use $D \geq 0$ to indicate that $\ell_{j} \geq 0$ for all $j \in \{1,\ldots,n\}$, and $D_{1}\geq D_{2}$ to indicate that $D_{1}-D_{2}\geq 0$. The divisor of a non-zero scalar valued meromorphic function $f$ on $\mathcal{M}$ is defined as $\mathrm{div} (f) = \sum_{\z \in \mathcal{Z}_{f}} \ell_{\z} \z - \sum_{\z \in \mathcal{Q}_{f}} \ell_{\z} \z$, where $\mathcal{Z}_{f}$ and $\mathcal{Q}_{f}$ are the finite sets of all zeros and poles of $f$, respectively, and $\ell_{\z} \in \mathbb{N}_{>0}$ is the order of $\z$. 
\end{definition}
\begin{definition}\label{def: zeros and poles}(poles and zeros) Given a matrix valued meromorphic function $F$ on $\mathcal{M}$, we say that $\z$ is a zero of $F$ (of order $m$) if $\z$ is a common zero of all entries of $F$ (of order $\geq m$ for each entry, and of order exactly $m$ for at least one entry). Similarly, $\z$ is a pole of $F$ (of order $m$) if $\z$ is a pole (of order $m$) for at least one entry of $F$ (and the other entries of $F$ have either no pole at $\z$, or a pole of order $\leq m$). Of course, if $F$ is scalar valued, the number of its poles equals the number of its zeros (counting multiplicities), but this is not true in general if $F$ is matrix valued. For instance, the function $\e$ of Example \ref{example: first} with $k \hspace{-0.15cm} \mod 2 = 1$ and $k>0$ has a pole of order $k$ at $\infty^{(1)}=\infty^{(2)}$ and no zero.
\end{definition}
\begin{definition} ($n_{\z}$ and $\widehat{n}_{\z}$)\label{def: nz and nz hat} Let $\mathcal{Z}$ and $\widehat{\mathcal{Z}}$ be the finite sets of all zeros of $\e$ and $\widehat{\e}$, respectively, and recall that $\mathcal{Q}$ and $\widehat{\mathcal{Q}}$ are the finite sets of all poles of $\e$ and $\widehat{\e}$, respectively. The order of a zero $\z \in \mathcal{Z}$ of $\e$ is denoted by $n_{\z}$, and the order of a pole $\z \in \mathcal{Q}$ of $\e$ is denoted by $-n_{\z}$. That is, $n_{\z}>0$ if $\z \in \mathcal{Z}$ and $n_{\z}<0$ if $\z \in \mathcal{Q}$. Similarly, for each $\z \in \widehat{\mathcal{Z}}\cup \widehat{\mathcal{Q}}$, we associate an integer $\widehat{n}_{\z} \in \mathbb{Z}\setminus \{0\}$, which represents the order of the zero, or the opposite of the order of the pole of $\widehat{\e}$ at $\z$.
\end{definition}
\begin{definition} ($\infty^{(j)}$)
Let $\infty^{(j)}$ be the point at infinity on the $j$-th sheet of $\mathcal{M}$. If on each sheet $\mathcal{M}$ has no branch point at infinity, then $\# \{\infty^{(1)},\ldots,\infty^{(r)}\} = r$; otherwise $\# \{\infty^{(1)},\ldots,\infty^{(r)}\} < r$. 
\end{definition}
Given a scalar polynomial $P \in \mathcal{P}_{N-1} := \mathcal{P}_{N-1}^{1\times 1}$, the function 
\begin{align*}
\z \mapsto P(z), \qquad \mbox{ where $z$ is the projection of $\z$ on $\widehat{\mathbb{C}}$},
\end{align*}
is meromorphic on $\mathcal{M}$ with a pole of order $(N-1)m_{j}$ at $\infty^{(j)}$, $j=1,\ldots,r$, where
\begin{align*}
m_{j} = \#\{\ell \in \{1,\ldots,r\}: \infty^{(\ell)} = \infty^{(j)}\}.
\end{align*}
Therefore, it is immediate to see from \eqref{def of LN and LN star} that
\vspace{-0.3cm}\begin{align}
& L_{N} \subseteq \{f: \mathrm{div}(f) \geq -\sum_{j=1}^{r}(N-1) \cdot \infty^{(j)} + \sum_{\z \in \mathcal{Z}\cup \mathcal{Q}} n_{\z}\z  \}, \label{LN as a subspace} \\
& \widehat{L}_{N} \subseteq \{f: \mathrm{div}(f) \geq -\sum_{j=1}^{r}(N-1) \cdot \infty^{(j)} + \sum_{\z \in \widehat{\mathcal{Z}}\cup \widehat{\mathcal{Q}}} \widehat{n}_{\z}\z \}. \label{LNstar as a subspace}
\end{align}
\begin{remark}
The opposite inclusion $\supseteq$ does not hold in general. To see this, consider Example \ref{example: first} with $k \hspace{-0.15cm} \mod 2=1$, $k>1$. Since $\eta$ has a pole of order $k$ at $\infty^{(1)}=\infty^{(2)}$, one has
\begin{align*}
L_{N} = \{(z,\eta)\mapsto P_{1}(z)+\eta P_{2}(z): P_{1},P_{2} \in \mathcal{P}_{N-1}\} \subseteq \{f : \mathrm{div}(f) \geq -(2(N-1)+k)\infty^{(1)} \},
\end{align*}
which is consistent with \eqref{LN as a subspace}. However, the function
\begin{align*}
(z,\eta) \mapsto \sqrt{z}= \frac{\eta}{z^{\frac{k-1}{2}}}
\end{align*}
has only a simple pole at $\infty^{(1)}$, but does not belong to $L_{N}$. In particular,
\begin{align*}
L_{N} \subsetneq \{f : \mathrm{div}(f) \geq -(2(N-1)+k)\infty^{(1)} \}.
\end{align*}

\end{remark}

\paragraph{Genus $0$ situation.} Our second main result, Theorem \ref{thm: reprod Rcal U}, states that if $\mathcal{M}$ is a connected Riemann surface of genus $0$, $\mathcal{R}_{N}^{W}$ is in fact completely equivalent to the scalar reproducing kernel $\mathfrak{R}_{rN}^{\mathcal{W}}$ of certain polynomial spaces $\mathcal{V}$ and $\widehat{\mathcal{V}}$. This scalar kernel is associated to a scalar weight $\mathcal{W}$ on a contour $\gamma_{\mathbb{C}} \subset \mathbb{C}$ that are described below. In certain situations, it holds that $\mathcal{V} = \widehat{\mathcal{V}} = \mathcal{P}_{rN-1}$, in which case $\mathfrak{R}_{rN}^{\mathcal{W}}$ is equal to the scalar CD kernel $\mathcal{R}_{rN}^{\mathcal{W}}$. However, we emphasize that in general $\mathfrak{R}_{rN}^{\mathcal{W}}$ is \textit{not} necessarily a CD kernel. We now define the relevant quantities that will appear in the statement of Theorem \ref{thm: reprod Rcal U}.

\medskip \noindent We assume from now that $\mathcal{M}$ is a connected Riemann surface of genus $0$.
\begin{definition}\label{def: varphi} (the maps $\varphi$ and $\phi$) 
Since $\mathcal{M}$ is of genus $0$, there is a one-to-one map $\zeta \mapsto \varphi(\zeta)$ from $\widehat{\mathbb{C}}$ to $\mathcal{M}$. The projection of $\varphi(\zeta)$ to the complex plane will be denoted by $\phi(\zeta)$.
\end{definition}
\begin{definition} (the scalar kernel $\mathfrak{R}_{rN}^{\mathcal{W}}$)
We define $\mathfrak{R}_{rN}^{\mathcal{W}}$ by
\begin{align}
\mathfrak{R}_{rN}^{\mathcal{W}}(\omega,\zeta) & = \widehat{h}(\omega)R_{N}^{\lambda}(\varphi(\omega),\varphi(\zeta))h(\zeta), \nonumber \\
 & = \widehat{h}(\omega)\widehat{\e}(\varphi(\omega))\mathcal{R}_{N}^{W}(\phi(\omega),\phi(\zeta))\e(\varphi(\zeta))h(\zeta), \quad \zeta, \omega \in \mathbb{C}, \label{def of RU}
\end{align}
where $h$ and $\widehat{h}$ are the two scalar valued rational functions given by
\begin{align}
& h(\zeta) = \prod_{\substack{j=1 \\ \infty^{(j)} \neq \varphi(\infty) }}^{r}(\zeta- \varphi^{-1}(\infty^{(j)}))^{N-1} \prod_{\z \in (\mathcal{Z}\cup \mathcal{Q})\setminus \{\varphi(\infty)\}}(\zeta-\varphi^{-1}(\z))^{-n_{\z}}, \label{def of h} \\
& \widehat{h}(\zeta) = \prod_{\substack{j=1 \\ \infty^{(j)} \neq \varphi(\infty) }}^{r}(\zeta- \varphi^{-1}(\infty^{(j)}))^{N-1} \prod_{\z \in (\widehat{\mathcal{Z}}\cup \widehat{\mathcal{Q}})\setminus \{\varphi(\infty)\}}(\zeta-\varphi^{-1}(\z))^{-\widehat{n}_{\z}}. \label{def of h hat}
\end{align}
The role of $h$ and $\widehat{h}$ is to remove all the poles in $\mathbb{C}$ of 
\begin{align*}
\zeta \mapsto \mathcal{R}_{N}^{W}(\phi(\omega),\phi(\zeta))\e(\varphi(\zeta)) \qquad \mbox{ and } \qquad \omega \mapsto \widehat{\e}(\varphi(\omega))\mathcal{R}_{N}^{W}(\phi(\omega),\phi(\zeta)),
\end{align*}
and to remove the zeros of $\zeta \to \e(\varphi(\zeta))$ and $\omega \to \widehat{\e}(\varphi(\omega))$. In particular, $h$ and $\widehat{h}$ are such that
\begin{align*}
\zeta \mapsto \mathcal{R}_{N}^{W}(\phi(\omega),\phi(\zeta))\e(\varphi(\zeta))h(\zeta) \qquad \mbox{ and } \qquad \omega \mapsto \widehat{h}(\omega)\widehat{\e}(\varphi(\omega))\mathcal{R}_{N}^{W}(\phi(\omega),\phi(\zeta))
\end{align*}
are polynomials with no prescribed zeros.
\end{definition}
\begin{definition}\label{def: scalar weight} ($\mathcal{W}$ and $\gamma_{\mathbb{C}}$)
The contour $\gamma_{\mathbb{C}}$ is defined by $\gamma_{\mathbb{C}} = \varphi^{-1}(\gamma_{\mathcal{M}}) = \varphi^{-1}(\cup_{j=1}^{r} \gamma^{(j)})$, and the scalar weight $\mathcal{W}$ by
\begin{align}\label{scalar weight}
& \mathcal{W}(\omega) = \frac{\lambda(\varphi(\omega))}{h(\omega)\widehat{h}(\omega)} \phi'(\omega).
\end{align}
\end{definition}
\begin{definition}\label{def: V and V*} (The polynomial spaces $\mathcal{V}$ and $\widehat{\mathcal{V}}$) We define
\begin{align}\label{V and V*}
\mathcal{V} = \{\zeta \mapsto p(\zeta)=f(\varphi(\zeta)) h(\zeta): f \in L_{N}\}, \qquad \widehat{\mathcal{V}} = \{\zeta \mapsto p(\zeta)=f(\varphi(\zeta)) \widehat{h}(\zeta): f \in \widehat{L}_{N}\}.
\end{align}
\end{definition}
\begin{remark}\label{remark: V and V* are polynomial spaces}
If $P \in \mathcal{P}_{N-1}$, then $\zeta \mapsto P(\varphi(\zeta))$ is a meromorphic function on $\widehat{\mathbb{C}}$ with poles of order $(N-1)m_{j}$ at $\varphi^{-1}(\infty^{(j)})$, $j=1,...,r$. Hence, by \eqref{def of LN and LN star} and \eqref{def of h}--\eqref{def of h hat}, the elements of $\mathcal{V}$ and $\widehat{\mathcal{V}}$ are polynomials of degree 
\begin{align*}
& \leq r(N-1)- \sum_{\z \in \mathcal{Z}\cup \mathcal{Q}} n_{\z} \qquad \mbox{ and } \qquad \leq r(N-1)- \sum_{\z \in \widehat{\mathcal{Z}}\cup \widehat{\mathcal{Q}}} \widehat{n}_{\z}, \qquad \mbox{ respectively}.
\end{align*}
Also, since $\mathcal{V}$ is in bijection with $L_{N}$ via the obvious map $[\z \mapsto f(\z)] \mapsto [\zeta \mapsto f(\varphi(\zeta)) h(\zeta)]$, we have $\dim \mathcal{V} = rN$. Similarly, $\widehat{\mathcal{V}}$ is in bijection with $\widehat{L}_{N}$ and $\dim \widehat{\mathcal{V}} = rN$. In particular, we always have
\begin{align}\label{some bounds}
-\sum_{\z \in \mathcal{Z}\cup \mathcal{Q}} n_{\z} \geq r-1 \qquad \mbox{ and } \qquad -\sum_{\z \in \widehat{\mathcal{Z}}\cup \widehat{\mathcal{Q}}}\widehat{n}_{\z} \geq r-1,
\end{align}
and $\mathcal{V} = \mathcal{P}_{rN-1}$ (resp. $\widehat{\mathcal{V}} = \mathcal{P}_{rN-1}$) if and only if $-\sum_{\z \in \mathcal{Z}\cup \mathcal{Q}} n_{\z} = r-1$ (resp. $-\sum_{\z \in \widehat{\mathcal{Z}}\cup \widehat{\mathcal{Q}}}\widehat{n}_{\z} = r-1$).
\end{remark}


\begin{theorem}\label{thm: reprod Rcal U}
Let $\gamma \subset \mathbb{C}$ be a finite union of piecewise smooth, oriented curves, let $W$ be a rational $r \times r$ matrix weight, and let $N \in \mathbb{N}_{>0}$. Suppose that $W$ satisfies Assumption \ref{ass: weight diag} and that $\mathcal{R}_{N}^{W}$ exists. Let $\mathcal{M}$ be an $r$-sheeted Riemann surface such that the functions $\lambda$, $\e$ and $\widehat{\e}$ defined in \eqref{def of lambda RS}--\eqref{def of efrak inv RS} are meromorphic, and assume that $\mathcal{M}$ is of genus $0$. Define $\mathfrak{R}_{rN}^{\mathcal{W}}$, $\varphi$, $\phi$, $\mathcal{W}$, $\gamma_{\mathbb{C}}$, $\mathcal{V}$ and $\widehat{\mathcal{V}}$ as in Definitions \ref{def: varphi}--\ref{def: V and V*}. The scalar kernel $\mathfrak{R}_{rN}^{\mathcal{W}}$ exists, $\dim \mathcal{V} = \dim \widehat{\mathcal{V}}=rN$ and we have:
\begin{itemize}
\item[(a)] $\zeta \mapsto \mathfrak{R}_{rN}^{\mathcal{W}}(\omega,\zeta) \in \mathcal{V}$ for every $\omega \in \mathbb{C}$, 
\item[(b)] $\omega \mapsto \mathfrak{R}_{rN}^{\mathcal{W}}(\omega,\zeta) \in \widehat{\mathcal{V}}$ for every $\zeta \in \mathbb{C}$, 
\item[(c)] $\mathfrak{R}_{rN}^{\mathcal{W}}$ satisfies the following reproducing property for $\mathcal{V}$:
\begin{align}\label{reproducing property C}
& \int_{\gamma_{\mathbb{C}}} p(\omega) \mathcal{W}(\omega) \mathfrak{R}_{rN}^{\mathcal{W}}(\omega,\zeta) d\omega = p(\zeta), & & \mbox{for every } p \in \mathcal{V} \mbox{ and } \zeta \in \mathbb{C}.
\end{align}
\item[(d)] $\mathfrak{R}_{rN}^{\mathcal{W}}$ satisfies the following reproducing property for $\widehat{\mathcal{V}}$:
\begin{align}\label{reproducing property C 2}
& \int_{\gamma_{\mathbb{C}}}  \mathfrak{R}_{rN}^{\mathcal{W}}(\omega,\zeta) \mathcal{W}(\zeta) p(\zeta) d\zeta = p(\omega), & & \mbox{for every } p \in \widehat{\mathcal{V}} \mbox{ and } \omega \in \mathbb{C}.
\end{align}
\item[(e)] We have the following equivalences
\begin{align*}
-\sum_{\z \in \mathcal{Z}\cup \mathcal{Q}} n_{\z} = r-1 \; \Leftrightarrow \; -\sum_{\z \in \widehat{\mathcal{Z}}\cup \widehat{\mathcal{Q}}}\widehat{n}_{\z} = r-1 \; \Leftrightarrow \; \mathcal{V} = \mathcal{P}_{rN-1} \; \Leftrightarrow \; \widehat{\mathcal{V}} = \mathcal{P}_{rN-1} \; \Leftrightarrow \; \mathfrak{R}_{rN}^{\mathcal{W}} = \mathcal{R}_{rN}^{\mathcal{W}},
\end{align*}
where $\mathcal{R}_{rN}^{\mathcal{W}}$ denotes the scalar CD kernel associated to $\mathcal{W}$.
\end{itemize}
\end{theorem}
\begin{proof}
See Section \ref{subsection: proof thm}.
\end{proof}
We now illustrate how to apply Theorem \ref{thm: reprod Rcal U} in some concrete cases. The following example considers a situation where $\mathcal{M}$ is of genus $0$, but the scalar kernel $\mathfrak{R}_{rN}^{\mathcal{W}}$ is not equivalent to a CD kernel (except for a particular choice of the parameters).
\begin{example}\label{example: scalar 1}
Consider the matrix weight
\begin{align*}
W(z) =  \frac{1}{z^{M}}\begin{pmatrix}
1 & 1 \\ z^{k} & 1
\end{pmatrix}^{L}, \qquad L,M \in \mathbb{N}_{>0},
\end{align*}
and $k \in \mathbb{N}_{>0}$ is odd, and let $\gamma \subset \mathbb{C}$ be such that $0 \notin \gamma$. This weight is similar but slightly more complicated than in Example \ref{example: first}. For each $z \in \mathbb{C}\setminus \{0\}$, we can write $W(z) = E(z) \Lambda(z) E(z)^{-1}$ with
\begin{align*}
E(z)= \begin{pmatrix}
1 & 1 \\ z^{\frac{k}{2}} & -z^{\frac{k}{2}}
\end{pmatrix}, \qquad \Lambda(z) = z^{-M}\begin{pmatrix}
(1+z^{\frac{k}{2}})^{L} & 0 \\
0 & (1-z^{\frac{k}{2}})^{L}
\end{pmatrix}, \qquad E(z)^{-1}= \frac{1}{2}\begin{pmatrix}
1 & z^{-\frac{k}{2}} \\
1 & -z^{-\frac{k}{2}}
\end{pmatrix},
\end{align*}
where the principal branch is chosen for $z^{\frac{k}{2}}$ (for $z<0$, one needs to replace all $z^{\frac{k}{2}}$'s above by $|z|^{\frac{k}{2}}e^{\frac{\pi i k}{2}}$). In particular, $W$ satisfies Assumption \ref{ass: weight diag}. The Riemann surface $\mathcal{M}$ associated to 
\begin{align*}
\{(z,\eta)\in \mathbb{C}^{2}:\eta^{2} = z^{k}\}
\end{align*}
is of genus $0$, and the functions 
\begin{align}\label{lambde e and einv in first example}
\lambda((z,\eta)) = z^{-M}(1+\eta)^{L}, \qquad \e((z,\eta)) = \begin{pmatrix}
1 & \eta
\end{pmatrix}^{T}, \qquad \widehat{\e}((z,\eta)) = \frac{1}{2}\begin{pmatrix}
1 & \eta^{-1}
\end{pmatrix}
\end{align}
are meromorphic on $\mathcal{M}$. Therefore we know by Theorem \ref{thm: reprod Rcal U} that the matrix CD kernel $\mathcal{R}_{N}^{W}$ (assuming it exists) is equivalent to a scalar reproducing kernel $\mathfrak{R}_{rN}^{\mathcal{W}}$ which can be explicitly described as follows. First, we note that $\e$ has no zero and a pole of order $k$ at $\infty^{(1)}=\infty^{(2)}$, and $\widehat{\e}$ has no zero and a pole of order $k$ at $0^{(1)}=0^{(2)}$. Therefore, following Definition \ref{def: nz and nz hat}, we have
\begin{align*}
\mathcal{Z} = \emptyset, \qquad \mathcal{Q} = \{\infty^{(1)}\}, \qquad \widehat{\mathcal{Z}} = \emptyset, \qquad \widehat{\mathcal{Q}} = \{0^{(1)}\}, \qquad n_{\infty^{(1)}} = -k, \qquad \widehat{n}_{0^{(1)}} = -k.
\end{align*}
Define the bijections $\varphi : \widehat{\mathbb{C}} \to \mathcal{M}$ and $\varphi^{-1} : \mathcal{M} \to \widehat{\mathbb{C}}$ by
\begin{align}\label{lol2}
& \varphi(\zeta) = (\zeta^{2},\zeta^{k}), & & \varphi^{-1}((z,\eta))=\eta z^{-\frac{k-1}{2}}=\sqrt{z}.
\end{align}
The functions $\lambda \circ \varphi$, $\e \circ \varphi$ and $\widehat{\e}\circ \varphi$ can be computed from \eqref{lambde e and einv in first example} and \eqref{lol2}, and we obtain
\begin{align*}
\lambda(\varphi(\zeta)) = \zeta^{-2M}(1+\zeta^{k})^{L}, \qquad \e(\varphi(\zeta)) = \begin{pmatrix}
1 & \zeta^{k}
\end{pmatrix}^{T}, \qquad \widehat{\e}(\varphi(\zeta)) = \frac{1}{2}\begin{pmatrix}
1 & \zeta^{-k}
\end{pmatrix}.
\end{align*}
Since $\varphi(\infty) = \infty^{(1)}$ and $\varphi(0) = 0^{(1)}$, by \eqref{def of h}-\eqref{def of h hat} we have $h(\zeta) = 1$ and $\widehat{h}(\zeta) = \zeta^{k}$. Therefore, from \eqref{V and V*} and \eqref{def of LN and LN star}, we have
\begin{align}\label{lol1}
\mathcal{V} = \widehat{\mathcal{V}} = \{P_{1}(\zeta^{2}) + \zeta^{k}P_{2}(\zeta^{2}) : P_{1},P_{2} \in \mathcal{P}_{N-1}\}  \qquad \begin{array}{l l}
= \mathcal{P}_{2N-1}, & \mbox{if } k=1, \\
\neq \mathcal{P}_{2N-1}, & \mbox{if } k>1.
\end{array}
\end{align}
Hence, by criteria (e) of Theorem \ref{thm: reprod Rcal U}, $\mathfrak{R}_{2N}^{\mathcal{W}}$ is a CD kernel if and only if $k=1$. Also, since $\phi(\zeta) = \zeta^{2}$, we infer from \eqref{scalar weight} that 
\begin{align*}
& \mathcal{W}(\zeta) = 2\zeta^{-2M-k+1}(1+\zeta^{k})^{L}, & & \gamma_{\mathbb{C}} = \varphi^{-1}(\gamma_{\mathcal{M}}) = \varphi^{-1}(\cup_{j=1}^{r} \gamma^{(j)}).
\end{align*}
It is easy to verify that if $\gamma$ is the unit circle, then $\gamma_{\mathbb{C}} = \gamma$. On the other hand, if $\gamma = (a,b)$ for certain $b>a>0$, then $\gamma_{\mathbb{C}} = (-\sqrt{b},-\sqrt{a}) \cup (\sqrt{a},\sqrt{b})$; this provides an example of a ``one-cut" matrix orthogonality that leads to a ``two-cuts" scalar orthogonality. It is also interesting to note that if $k=1$, $\mathcal{W}$ above is nothing else than the Jacobi weight with non-standard parameters.  
\end{example}
The above example only deals with $r=2$. Example \ref{example: scalar 2} below provides an application of Theorem \ref{thm: reprod Rcal U} in a situation where the matrix weight is of size $r \times r$, where $r \in \mathbb{N}_{>0}$ is arbitrary.
\begin{example}\label{example: scalar 2}
Consider the weight $W$ given by \eqref{W example rxr}, and let $\gamma \subset \mathbb{C}$ be such that $0 \notin \gamma$. A simple computation shows that
\begin{align*}
W(z) =  E(z) \frac{\widehat{\Lambda}(z)^{L}}{z^{R}} E(z)^{-1}, \qquad z \in \mathbb{C}\setminus \{0\},
\end{align*}
with
\begin{align*}
E(z) = \begin{pmatrix}
1 & 1 & \cdots & 1 \\
z^{\frac{1}{r}} & \rho_{r} z^{\frac{1}{r}} & \cdots & \rho_{r}^{r-1}z^{\frac{1}{r}} \\
\vdots & \vdots & \ddots & \vdots \\
z^{\frac{r-1}{r}} & \rho_{r}^{r-1} z^{\frac{r-1}{r}} & \cdots & \rho_{r}^{(r-1)^{2}}z^{\frac{r-1}{r}}
\end{pmatrix}, \quad E(z)^{-1} = \frac{1}{r} \begin{pmatrix}
1 & z^{-\frac{1}{r}} & \cdots & z^{-\frac{r-1}{r}} \\
1 & \rho_{r} z^{-\frac{1}{r}} & \cdots & \rho_{r}^{r-1} z^{-\frac{r-1}{r}} \\
\vdots & \vdots & \ddots & \vdots \\
1 & \rho_{r}^{r-1}z^{-\frac{1}{r}} & \cdots & \rho_{r}^{(r-1)^{2}} z^{-\frac{r-1}{r}}
\end{pmatrix},
\end{align*}
$\widehat{\Lambda}(z) = \mbox{diag}(1+z^{\frac{1}{r}},1+\rho_{r}z^{\frac{1}{r}}, \ldots, 1+\rho_{r}^{r-1}z^{\frac{1}{r}})$ and $\rho_{r} = e^{\frac{2\pi i}{r}}$, and where the principal branches are taken for the roots (for $z<0$, one needs to replace all $z^{\frac{1}{r}}$'s above by $|z|^{\frac{1}{r}}e^{\frac{\pi i}{r}}$). In particular, $W$ satisfies Assumption \ref{ass: weight diag}. The $r$-sheeted Riemann surface $\mathcal{M}$ associated to $\{(z,\eta)\in \mathbb{C}^{2} : \eta^{r} = z \}$ is of genus $0$. For convenience, the sheets are numbered such that $\eta \sim \rho_{r}^{k-1}z^{\frac{1}{r}}$ as $z \to + \infty$ on the $k$-th sheet, $k=1,\ldots,r$. The functions $\e$, $\widehat{\e}$ and $\lambda$, defined in \eqref{def of lambda RS}--\eqref{def of efrak inv RS}, are given by
\begin{align*}
\e((z,\eta))^{T} = \begin{pmatrix}
1 & \eta & \eta^{2} &
\cdots & \eta^{r-1}
\end{pmatrix}, \quad \widehat{\e}((z,\eta)) = \frac{1}{r} \begin{pmatrix}
1 & \eta^{-1} & \ldots & \eta^{-(r-1)}
\end{pmatrix}, \quad \lambda((z,\eta)) = \frac{(1+ \eta)^{L}}{z^{R}}
\end{align*}
and are meromorphic on $\mathcal{M}$. Since $\mathcal{M}$ is of genus $0$, Theorem \ref{thm: reprod Rcal U} implies that $\mathcal{R}_{N}^{W}$ is equivalent to a scalar kernel $\mathfrak{R}_{rN}^{\mathcal{W}}$ which can be described as follows. Consider the natural bijections $\varphi: \widehat{\mathbb{C}} \to \mathcal{M}$, $\varphi^{-1}: \mathcal{M} \to \widehat{\mathbb{C}}$ given by 
\begin{align*}
& \varphi(\zeta) = (\zeta^{r},\zeta), & & \varphi^{-1}((z,\eta)) = \eta.
\end{align*}
Since $\e$ has no zero and a pole of order $r-1$ at $\infty^{(1)} = \ldots = \infty^{(r)} = (\infty,\infty)$, by \eqref{def of h} and Definition \ref{def: nz and nz hat}, we have 
\begin{align*}
\mathcal{Z} = \emptyset, \qquad \mathcal{Q} = \{\infty^{(1)}\}, \qquad n_{\infty^{(1)}} = -(r-1), \qquad h(\zeta)=1.
\end{align*}
Similarly, $\widehat{\e}$ has no zero and a pole of order $r-1$ at $0^{(1)} = \ldots = 0^{(r)} = (0,0)$, and thus from \eqref{def of h hat} we get
\begin{align*}
\widehat{\mathcal{Z}} = \emptyset, \qquad \widehat{\mathcal{Q}} = \{0^{(1)}\}, \qquad \widehat{n}_{0^{(1)}} = -(r-1), \qquad \widehat{h}(\zeta) = (\zeta-\varphi^{-1}(0^{(1)}))^{r-1} = \zeta^{r-1}.
\end{align*}
Since $-\sum_{\z \in \mathcal{Z}\cup \mathcal{Q}} n_{\z} = r-1$, we known from criteria (e) of Theorem \ref{thm: reprod Rcal U} that $\mathcal{V}=\widehat{\mathcal{V}}=\mathcal{P}_{rN-1}$. This fact can also be verified directly from the definition \eqref{V and V*}:
\begin{align*}
\mathcal{V} = \widehat{\mathcal{V}} = \{P_{1}(\zeta^{r}) + \zeta P_{2}(\zeta^{r}) + \ldots + \zeta^{r-1}P_{r}(\zeta^{r}) : P_{1},\ldots,P_{r} \in \mathcal{P}_{N-1}\} = \mathcal{P}_{rN-1}.
\end{align*}
Therefore $\mathfrak{R}_{rN}^{\mathcal{W}} = \mathcal{R}_{rN}^{\mathcal{W}}$ is the scalar CD kernel associated to the scalar weight 
\begin{align}\label{weight of example 1.18}
& \mathcal{W}(\zeta) = \frac{\lambda(\varphi(\zeta))}{h(\zeta)\widehat{h}(\zeta)} \phi'(\zeta) = r\zeta^{-rR}(1+\zeta)^{L}, \qquad \gamma_{\mathbb{C}} = \varphi^{-1}(\gamma_{\mathcal{M}}) = \varphi^{-1}(\cup_{j=1}^{r} \gamma^{(j)}).
\end{align}
If $\gamma$ is the unit circle, the existence of $\mathcal{R}_{N}^{W}$ is guaranteed from the general result \cite[Lemma 4.8]{DK} (see also Section \ref{section: applications}), and it is easy to see that $\gamma_{\mathbb{C}} = \gamma$. On the other hand, if $\gamma = (a,b)$ for certain $b>a>0$, then $\gamma_{\mathbb{C}} = \cup_{j=1}^{r} \rho_{r}^{j-1}(a^{\frac{1}{r}},b^{\frac{1}{r}})$ is a disjoint union of $r$ intervals in the complex plane.
\end{example}

We provide other applications of Theorem \ref{thm: reprod Rcal U} in Section \ref{section: applications}.

\medskip It is well-known that matrix valued CD kernels can be expressed in terms of MOPs and are related to certain Riemann-Hilbert (RH) problems of large size \cite{Dur CD formula,DPS, Delvaux, GrunIglesia,DK}. Theorem \ref{thm: reprod Rcal U} (e) establishes a link between matrix and scalar CD kernels, and therefore it also admits two reformulations -- one involving MOPs, the other one involving RH problems -- which we believe are of interest. 


\subsection{Theorem \ref{thm: reprod Rcal U} (e) in terms of MOPs}\label{subsubsection: from matrix to scalar OPs}
In the context of matrix orthogonality associated with an inner product, it is well-known that the reproducing kernel can be expressed in terms of MOPs by mean of a so-called CD formula, see e.g. \cite{Dur CD formula,DPS, Delvaux, GrunIglesia}. The adaptation of this formula to our setting only requires minor modifications, which we present here. 

\medskip As mentioned in the introduction, the pairing \eqref{pairing} induces two families of MOPs. 

\medskip \noindent We recall that $P_{j}^{\mathrm{L}}(z) = z^{j}I_{r}+...$ and 
$P_{j}^{\mathrm{R}}(z) = z^{j}I_{r}+...$ are the two degree $j$ MOPs that satisfy
\begin{align}\label{orthogonality for PNR and PNL}
& \langle P_{j}^{\mathrm{L}},z^{k}I_{r} \rangle = 0_{r}, & & \langle z^{k}I_{r}, P_{j}^{\mathrm{R}} \rangle =0_{r}, & & k=0,\ldots,j-1,
\end{align}
and we let $Q_{j-1}^{\mathrm{L}}$ and $Q_{j-1}^{\mathrm{R}}$ denote the MOPs of degree $\leq j-1$ characterized by
\begin{align}\label{orthogonality for Q}
& \langle Q_{j-1}^{\mathrm{L}},z^{k}I_{r} \rangle = \delta_{k,j-1}I_{r}, & & \langle z^{k}I_{r},Q_{j-1}^{\mathrm{R}} \rangle = \delta_{k,j-1}I_{r}, & & k=0,\ldots,j-1.
\end{align}
Note that each of the four systems given in \eqref{orthogonality for PNR and PNL} and \eqref{orthogonality for Q} gives $r^{2}j$ equations for the $r^{2}j$ unknown scalar coefficients of $P_{j}^{\mathrm{L}}$, $P_{j}^{\mathrm{R}}$, $Q_{j-1}^{\mathrm{L}}$ or $Q_{j-1}^{\mathrm{R}}$. Furthermore, the square matrices associated to these four linear systems are identical to one another. Therefore, if we have existence and uniqueness for one polynomial among $P_{j}^{\mathrm{L}}$, $P_{j}^{\mathrm{R}}$, $Q_{j-1}^{\mathrm{L}}$ and $Q_{j-1}^{\mathrm{R}}$, this implies existence and uniqueness for the other three polynomials. 

\medskip In certain special situations, some of these four families of MOPs can be directly related to each other. For example:
\begin{itemize}
\item[\textbullet] \vspace{-0.1cm} If $W = W^{T}$ and if $P_{j}^{\mathrm{L}}$ exists and is unique, then $P_{j}^{\mathrm{L}} = (P_{j}^{\mathrm{R}})^{T}$ and $Q_{j-1}^{\mathrm{L}} = (Q_{j-1}^{\mathrm{R}})^{T}$.
\item[\textbullet] \vspace{-0.1cm} If $Q_{j}^{\mathrm{L}}(z) = \kappa_{j}^{\mathrm{L}}z^{j}+\ldots$ exists, and if $\kappa_{j}^{\mathrm{L}} \in \mathbb{C}^{r \times r}$ is invertible, then $P_{j}^{\mathrm{L}}$ exists and is given by $(\kappa_{j}^{\mathrm{L}})^{-1}Q_{j}^{\mathrm{L}}$. Similarly, if $Q_{j}^{\mathrm{R}}(z) = \kappa_{j}^{\mathrm{R}}z^{j}+...$ exists with $\kappa_{j}^{\mathrm{R}}$ invertible, then $P_{j}^{\mathrm{R}}$ exists as well and is given by $Q_{j}^{\mathrm{R}}(\kappa_{j}^{\mathrm{R}})^{-1}$. However, there is no guarantee in general that $\kappa_{j}^{\mathrm{L}}$ and $\kappa_{j}^{\mathrm{R}}$ are invertible.
\end{itemize}
Note however that if $W=W^{\dag}$ and if $P_{j}^{\mathrm{L}}$ exists and is unique, then in general we cannot conclude that $P_{j}^{\mathrm{L}} = (P_{j}^{\mathrm{R}})^{\dag}$ (this follows by a direct inspection of \eqref{orthogonality for PNR and PNL} and \eqref{pairing}).

\noindent If the polynomials $Q_{0}^{\mathrm{R}},Q_{1}^{\mathrm{R}},\ldots,Q_{N-1}^{\mathrm{R}}$ exist and are unique, then $\mathcal{R}_{N}^{W}$ exists as well and is given by
\begin{align}\label{CD kernel}
\mathcal{R}_{N}^{W}(w,z) = \sum_{j=0}^{N-1} Q_{j}^{\mathrm{R}}(w) P_{j}^{\mathrm{L}}(z) = \sum_{j=0}^{N-1} P_{j}^{\mathrm{R}}(w)Q_{j}^{\mathrm{L}}(z).
\end{align}
To verify the validity of the above formula, first note that the sequences of MOPs $\{P_{j}^{\mathrm{L}}(z)\}_{j \geq 0}$ and $\{Q_{j}^{\mathrm{R}}(z)\}_{j \geq 0}$ are biorthogonal with respect to \eqref{pairing}:
\begin{align*}
\langle P_{j}^{\mathrm{L}},Q_{k}^{\mathrm{R}} \rangle = \delta_{k,j}I_{r}, \qquad \mbox{for all } j, k \geq 0.
\end{align*}
Since any polynomial $P \in \mathcal{P}_{N-1}^{r\times r}$ can be represented as $P(z) = \sum_{j=1}^{N-1}C_{j}P_{j}^{\mathrm{L}}(z)$ for certain $C_{j} \in \mathbb{C}^{r \times r}$, one has
\begin{align*}
\langle P, \mathcal{R}_{N}^{W}(\cdot,z) \rangle =  \int_{\gamma} \bigg( \sum_{j=1}^{m}C_{j}P_{j}^{\mathrm{L}}(w) \bigg) W(w) \bigg( \sum_{k=0}^{N-1} Q_{k}^{\mathrm{R}}(w) P_{k}^{\mathrm{L}}(z) \bigg) dw = P(z).
\end{align*}
The above reproducing property is equivalent to \eqref{reproducing kernel}; hence the first formula in \eqref{CD kernel} holds by uniqueness of the CD kernel. The second formula in \eqref{CD kernel} can be proved in a similar way. As mentioned, the formulas \eqref{CD kernel} holds if and only if all polynomials $Q_{0}^{\mathrm{R}},Q_{1}^{\mathrm{R}},\ldots,Q_{N-1}^{\mathrm{R}}$ exist. In fact, the existence and uniqueness of $Q_{N-1}^{\mathrm{R}}$ alone ensures the existence of $\mathcal{R}_{N}^{W}$; this follows from the following CD formula:
\begin{align}\label{CD simplified}
\mathcal{R}_{N}^{W}(w,z) = \frac{1}{z-w}\Big( Q_{N-1}^{\mathrm{R}}(w)P_{N}^{\mathrm{L}}(z) - P_{N}^{\mathrm{R}}(w) Q_{N-1}^{\mathrm{L}}(z) \Big).
\end{align}
This equation was obtained using \cite[eq (4.33)]{DK}. In fact, the formula \cite[eq (4.33)]{DK} is written in terms of the solution to a certain RH problem (see also \eqref{def of mathcal R} below), and therefore, for the convenience of the reader, we sketch a proof of \eqref{CD simplified} in Appendix \ref{section: appendix}. 

\medskip We now present a reformulation of Theorem \ref{thm: reprod Rcal U} (e) in terms of MOPs.
\begin{theorem}
Let $\gamma \subset \mathbb{C}$ be a finite union of piecewise smooth, oriented curves, let $W$ be a rational $r \times r$ matrix weight, and let $N \in \mathbb{N}_{>0}$. Suppose that $W$ satisfies Assumption \ref{ass: weight diag} and that $\mathcal{R}_{N}^{W}$ exists. Let $\mathcal{M}$ be an $r$-sheeted Riemann surface such that the functions $\lambda$, $\e$ and $\widehat{\e}$ defined in \eqref{def of lambda RS}--\eqref{def of efrak inv RS} are meromorphic, and assume that $\mathcal{M}$ is of genus $0$. Define $\mathfrak{R}_{rN}^{\mathcal{W}}$, $h$, $\widehat{h}$, $\varphi$, $\phi$, $\mathcal{W}$, $\gamma_{\mathbb{C}}$, $\mathcal{V}$ and $\widehat{\mathcal{V}}$ as in Definitions \ref{def: varphi}--\ref{def: V and V*}. If $\mathcal{V} = \mathcal{P}_{rN-1}$, we have
\begin{align}\label{scalar kernel in terms of OPs}
\mathfrak{R}_{rN}^{\mathcal{W}}(\omega,\zeta) & =  \widehat{h}(\omega)\widehat{\e}(\varphi(\omega)) \frac{Q_{N-1}^{\mathrm{R}}(\phi(\omega))P_{N}^{\mathrm{L}}(\phi(\zeta)) - P_{N}^{\mathrm{R}}(\phi(\omega)) Q_{N-1}^{\mathrm{L}}(\phi(\zeta))}{\phi(\zeta)-\phi(\omega)} \e(\varphi(\zeta))h(\zeta) \nonumber \\
& = \frac{1}{\zeta-\omega}\Big( q_{rN-1}(\omega)p_{rN}(\zeta) - p_{rN}(\omega) q_{rN-1}(\zeta) \Big) = \mathcal{R}_{rN}^{\mathcal{W}}(\omega,\zeta),
\end{align}
where $p_{rN}(\zeta) = \zeta^{rN}+...$ is the monic scalar orthogonal polynomial defined by
\begin{align}
& \int_{\gamma_{\mathbb{C}}} p_{rN}(\zeta) \mathcal{W}(\zeta)\zeta^{k}d\zeta = 0, & & k = 0,\ldots,rN-1, \nonumber 
\end{align}
and $q_{rN-1}$ is the degree $\leq rN-1$ scalar orthogonal polynomial satisfying
\begin{align}
& \int_{\gamma_{\mathbb{C}}} q_{rN-1}(\zeta) \mathcal{W}(\zeta)\zeta^{k}d\zeta = \delta_{k,rN-1}, & & k = 0,\ldots,rN-1. \nonumber
\end{align}
\end{theorem}

\subsection{Theorem \ref{thm: reprod Rcal U} (e) in terms of solutions to RH problems}\label{subsubsection: from matrix to scalar RHP}
RH problems are boundary value problems for analytic functions; we refer to \cite{Deift} for an introduction, and to \cite{Bothner} for a recent historical review. It is well-known that scalar orthogonal polynomials can be characterized in terms of $2 \times 2$ RH problems, see \cite{FIK} (see also \cite{AvanA2004} for an equivalent formulation in terms of a scalar RH problem on a two-sheeted Riemann surface). The generalization of this result for MOPs has been studied in great detail in \cite{Delvaux, GrunIglesia, CassaManas2012, DK}. Let $\gamma_{0}$ be the contour $\gamma$ with all endpoints and points of self-intersection removed. The RH problem that is relevant for our setting is as follows.

\vspace{-0.3cm}
\subsubsection*{RH problem for $Y$}
\begin{itemize}
\item[(a)] $Y : \mathbb{C}\setminus \gamma \to \mathbb{C}^{2r \times 2r}$ is analytic.
\item[(b)] The limits of $Y(z)$ as $z$ approaches $\gamma_{0}$ from left and right exist, are continuous on $\gamma_{0}$, and are denoted by $Y_+$ and $Y_-$, respectively (here ``left" and ``right" refer to the orientation of $\gamma_{0}$). Furthermore, they are related by
\begin{equation*}
Y_{+}(z) = Y_{-}(z) \begin{pmatrix}
I_{r} & W(z) \\ 0_{r} & I_{r}
\end{pmatrix}, \hspace{0.5cm} \mbox{ for } z \in \gamma_{0}.
\end{equation*}
\item[(c)] As $z \to \infty$, we have $Y(z) = \left(I_{2r} + \bigO(z^{-1})\right) \begin{pmatrix}
z^{N}I_{r} & 0_{r} \\ 0_{r} & z^{-N}I_{r}
\end{pmatrix}$. \\[0.2cm]
As $z \to z_{\star} \in \gamma \setminus \gamma_{0}$, we have $Y(z) = \bigO(\log (z-z_{\star}))$.
\end{itemize}
The unique solution $Y(\cdot) = Y(\cdot;W,\gamma,N)$ to the above RH problem can be explicitly written in terms of MOPs (see also Appendix \ref{section: appendix}), and exists if and only if $P_{N}^{\mathrm{L}}$ exists and is unique. Furthermore, $Y$ satisfies $\det Y \equiv 1$, and therefore $Y^{-1}$ exists if and only if $Y$ exists. It follows from \cite[eq (4.33)]{DK} that the CD kernel $\mathcal{R}_{N}^{W}$ can be written as
\begin{align}\label{def of mathcal R}
\mathcal{R}_{N}^{W}(w,z) = \frac{1}{2\pi i(z-w)} \begin{pmatrix}
0_{r} & I_{r}
\end{pmatrix}Y^{-1}(w)Y(z)\begin{pmatrix}
I_{r} \\ 0_{r}
\end{pmatrix}.
\end{align}
Formula \eqref{def of mathcal R} is in fact equivalent to the CD formula \eqref{CD simplified}, see Appendix \ref{section: appendix} for details. 

\medskip Theorem \ref{thm: reprod Rcal U} (e) can also be reformulated as follows.
\begin{theorem}\label{thm: RHP}
Let $\gamma \subset \mathbb{C}$ be a finite union of piecewise smooth, oriented curves, let $W$ be a rational $r \times r$ matrix weight, and let $N \in \mathbb{N}_{>0}$. Suppose that $W$ satisfies Assumption \ref{ass: weight diag} and that $\mathcal{R}_{N}^{W}$ exists. Let $\mathcal{M}$ be an $r$-sheeted Riemann surface such that the functions $\lambda$, $\e$ and $\widehat{\e}$ defined in \eqref{def of lambda RS}--\eqref{def of efrak inv RS} are meromorphic, and assume that $\mathcal{M}$ is of genus $0$. Define $\mathfrak{R}_{rN}^{\mathcal{W}}$, $h$, $\widehat{h}$, $\varphi$, $\phi$, $\mathcal{W}$, $\gamma_{\mathbb{C}}$, $\mathcal{V}$ and $\widehat{\mathcal{V}}$ as in Definitions \ref{def: varphi}--\ref{def: V and V*}. If $\mathcal{V} = \mathcal{P}_{rN-1}$, we have
\begin{align}\label{scalar kernel in terms of RHP}
\mathfrak{R}_{rN}^{\mathcal{W}}(\omega,\zeta) & =  \widehat{h}(\omega)\widehat{\e}(\varphi(\omega)) \frac{\begin{pmatrix}
0_{r} & I_{r}
\end{pmatrix}Y^{-1}(\phi(\omega))Y(\phi(\zeta))\begin{pmatrix}
I_{r} \\ 0_{r}
\end{pmatrix}}{2\pi i(\phi(\zeta)-\phi(\omega))}  \e(\varphi(\zeta))h(\zeta) \nonumber \\
& = \frac{1}{2\pi i(\zeta-\omega)}\begin{pmatrix}
0 & 1
\end{pmatrix}\mathcal{Y}^{-1}(\omega)\mathcal{Y}(\zeta)\begin{pmatrix}
1 \\ 0
\end{pmatrix} = \mathcal{R}_{rN}^{\mathcal{W}}(\omega,\zeta),
\end{align}
where $\mathcal{Y}$ is the solution to the following RH problem.
\subsubsection*{RH problem for $\mathcal{Y}$}
\begin{itemize}
\item[(a)] $\mathcal{Y} : \mathbb{C}\setminus \gamma_{\mathbb{C}} \to \mathbb{C}^{2 \times 2}$ is analytic.
\item[(b)] The limits of $\mathcal{Y}(\zeta)$ as $\zeta$ approaches $\gamma_{\mathbb{C},0}:= \varphi^{-1}(\cup_{j=1}^{r} \gamma_{0}^{(j)})$ from left and right exist, are continuous on $\gamma_{\mathbb{C},0}$, and are denoted by $\mathcal{Y}_+$ and $\mathcal{Y}_-$, respectively. Furthermore, they are related by
\begin{equation*}
\mathcal{Y}_{+}(\zeta) = \mathcal{Y}_{-}(\zeta) \begin{pmatrix}
1 & \mathcal{W}(\zeta) \\ 0 & 1
\end{pmatrix}, \hspace{0.5cm} \mbox{ for } \zeta \in \gamma_{\mathbb{C},0}.
\end{equation*}
\item[(c)] As $\zeta \to \infty$, we have $\mathcal{Y}(\zeta) = \left(I_{2} + \bigO(\zeta^{-1})\right) \begin{pmatrix}
\zeta^{rN} & 0 \\ 0 & \zeta^{-rN}
\end{pmatrix}$.  \\[0.2cm]
As $\zeta \to \zeta_{\star} \in \gamma_{\mathbb{C}}\setminus \gamma_{\mathbb{C},0}$, we have $\mathcal{Y}(\zeta) = \bigO(\log ( \zeta-\zeta_{\star}))$.
\end{itemize}
\end{theorem}
\begin{remark}
The \cite{DeiftZhou} steepest descent method is a powerful tool for asymptotic analysis of RH problems, and is particularly well-developed for RH problems of size $2 \times 2$. In principle, it is also possible to implement this method on RH problems of larger sizes, but it represents in general a much more complicated task. Theorem \ref{thm: RHP} implies that, in a situation where $\mathcal{M}$ has genus $0$ and $\mathcal{V} = \mathcal{P}_{rN-1}$ (such as in Example \ref{example: scalar 2}), one can study asymptotic properties of $\mathcal{R}_{N}^{W}$ by means of a $2 \times 2$ RH problem instead of a $2r \times 2r$ RH problem. This fact has already been proved useful in \cite{Charlier} in a situation where $r=2$. 
\end{remark}
\begin{remark}
We have not been able to obtain an analogue of Theorem \ref{thm: RHP} in the case where $\mathcal{M}$ is of genus $\geq 1$. More generally, to extend the standard toolbox for analysis of orthogonal polynomials on curves from complex plane to Riemann surfaces is a highly promising direction of the research. For instance, the recent works \cite{Chirka1,Chirka2,Chirka3} use modern results in multivariate complex analysis to develop such notions as
equilibrium measures, capacities and Green’s functions for compact Riemann surfaces. We also mention that, after the first version of this paper appeared, \cite{Bertola1,Bertola2} developed a non-linear steepest descent method to orthogonality on elliptic curves. 
\end{remark}

\section{Applications to tiling models}\label{section: applications}
Tiling models with doubly periodic weightings form a class of determinantal point processes with new interesting features \cite{CY2014, CJ, BCJ, DK, BeD, CDKL, Berggren, Charlier, BCJ2}. As it turns out, the correlation structure of these models can be studied by means of certain double contour integrals \cite{DK} which involve a matrix valued CD kernel. In this section, we use Theorems \ref{thm: reproducing kernel Riemann surface} and \ref{thm: reprod Rcal U} to simplify this formula. We first present the necessary material to invoke the formula from \cite{DK}. To simplify the presentation, we are going to focus on \textit{lozenge tilings of a hexagon}. We mention however that the main theorem of \cite{DK}, and therefore also Theorem \ref{thm: double contour integrals simplifications} below, can be applied to various other tiling models. 

%



\paragraph{Lozenge tilings of a hexagon.} Consider the infinite graph $\mathcal{G}_{H}$ whose vertex set is $\mathbb{Z}\times \mathbb{Z}$, and whose edges are of the form $\big( (x,y),(x+1,y) \big)$ or $\big( (x,y),(x+1,y+1) \big)$. A weighting on $\mathcal{G}_{H}$ consists of assigning to each edge a positive number. Here, we consider weightings that are $r \times q$ periodic, which means periodic of period $r$ in the vertical direction, and periodic of period $q$ in the horizontal direction. More precisely, an $r \times q$ periodic weighting depends on $2 r q$ edge weights, denoted by $a_{\ell,j}$, $b_{\ell,j}$, $0 \leq \ell \leq q-1$, $0 \leq j \leq r-1$, which we assign as follows:
\begin{align}\label{weight on a block}
\mbox{weight of } \big( (\ell,j),(\ell+1,j + \delta ) \big) = \begin{cases}
b_{\ell,j}, & \mbox{if } \delta = 0, \\
a_{\ell,j}, & \mbox{if } \delta = 1,
\end{cases} \qquad 0 \leq \ell \leq q-1, \; 0 \leq j \leq r-1.
\end{align}
Then, the weighting is extended over all edges of $\mathcal{G}_{H}$ by
\begin{align}\label{periodic weightings}
\hspace{-0.3cm}\mbox{weight of } \big( (\ell + m_{1}q,j +m_{2}r),(\ell+ m_{1}q +1,j +m_{2}r + \delta ) \big) = \mbox{weight of } \big( (\ell,j),(\ell+1,j + \delta ) \big),
\end{align}
for all $m_{1},m_{2} \in \mathbb{Z}$, and $\delta \in \{0,1\}$. The situation is illustrated in Figure \ref{fig:periodicity} for $r=2$ and $q=3$. 

\begin{figure}[h]
\begin{center}
\begin{tikzpicture}[master]
\node at (0,0) {};

\draw[fill] (0,0) circle (0.4mm);
\draw[fill] (0.5,0) circle (0.4mm);
\draw[fill] (1,0) circle (0.4mm);
\draw[fill] (1.5,0) circle (0.4mm);
\draw[fill] (2,0) circle (0.4mm);
\draw[fill] (2.5,0) circle (0.4mm);
\draw[fill] (3,0) circle (0.4mm);
\draw[fill] (3.5,0) circle (0.4mm);

\draw[fill] (0,0.5) circle (0.4mm);
\draw[fill] (0.5,0.5) circle (0.4mm);
\draw[fill] (1,0.5) circle (0.4mm);
\draw[fill] (1.5,0.5) circle (0.4mm);
\draw[fill] (2,0.5) circle (0.4mm);
\draw[fill] (2.5,0.5) circle (0.4mm);
\draw[fill] (3,0.5) circle (0.4mm);
\draw[fill] (3.5,0.5) circle (0.4mm);

\draw[fill] (0,1) circle (0.4mm);
\draw[fill] (0.5,1) circle (0.4mm);
\draw[fill] (1,1) circle (0.4mm);
\draw[fill] (1.5,1) circle (0.4mm);
\draw[fill] (2,1) circle (0.4mm);
\draw[fill] (2.5,1) circle (0.4mm);
\draw[fill] (3,1) circle (0.4mm);
\draw[fill] (3.5,1) circle (0.4mm);

\draw[fill] (0,1.5) circle (0.4mm);
\draw[fill] (0.5,1.5) circle (0.4mm);
\draw[fill] (1,1.5) circle (0.4mm);
\draw[fill] (1.5,1.5) circle (0.4mm);
\draw[fill] (2,1.5) circle (0.4mm);
\draw[fill] (2.5,1.5) circle (0.4mm);
\draw[fill] (3,1.5) circle (0.4mm);
\draw[fill] (3.5,1.5) circle (0.4mm);

\draw[fill] (0,2) circle (0.4mm);
\draw[fill] (0.5,2) circle (0.4mm);
\draw[fill] (1,2) circle (0.4mm);
\draw[fill] (1.5,2) circle (0.4mm);
\draw[fill] (2,2) circle (0.4mm);
\draw[fill] (2.5,2) circle (0.4mm);
\draw[fill] (3,2) circle (0.4mm);
\draw[fill] (3.5,2) circle (0.4mm);

\draw[fill] (0,2.5) circle (0.4mm);
\draw[fill] (0.5,2.5) circle (0.4mm);
\draw[fill] (1,2.5) circle (0.4mm);
\draw[fill] (1.5,2.5) circle (0.4mm);
\draw[fill] (2,2.5) circle (0.4mm);
\draw[fill] (2.5,2.5) circle (0.4mm);
\draw[fill] (3,2.5) circle (0.4mm);
\draw[fill] (3.5,2.5) circle (0.4mm);


\draw (0,0)--(3.5,0);
\draw (0,0.5)--(3.5,0.5);
\draw (0,1)--(3.5,1);
\draw (0,1.5)--(3.5,1.5);
\draw (0,2)--(3.5,2);
\draw (0,2.5)--(3.5,2.5);


\draw (0,0)--(2.5,2.5);

\draw (0.5,0)--(3,2.5);
\draw (1,0)--(3.5,2.5);
\draw (1.5,0)--(3.5,2);
\draw (2,0)--(3.5,1.5);
\draw (2.5,0)--(3.5,1);
\draw (3,0)--(3.5,0.5);

\draw (0,0.5)--(2,2.5);
\draw (0,1)--(1.5,2.5);
\draw (0,1.5)--(1,2.5);
\draw (0,2)--(0.5,2.5);
\draw (0,2.5)--(0,2.5);

\draw[dashed] (-0.5,0)--(4,0);
\draw[dashed] (-0.5,1)--(4,1);
\draw[dashed] (-0.5,2)--(4,2);
\draw[dashed] (0,-0.3)--(0,2.8);
\draw[dashed] (1.5,-0.3)--(1.5,2.8);
\draw[dashed] (3,-0.3)--(3,2.8);

\end{tikzpicture} \hspace{1cm} \begin{tikzpicture}[slave]
\draw[fill] (0,0) circle (0.6mm);
\draw[fill] (1,0) circle (0.6mm);
\draw[fill] (2,0) circle (0.6mm);
\draw[fill] (3,0) circle (0.6mm);

\draw[fill] (0,1) circle (0.6mm);
\draw[fill] (1,1) circle (0.6mm);
\draw[fill] (2,1) circle (0.6mm);
\draw[fill] (3,1) circle (0.6mm);

\draw[fill] (1,2) circle (0.6mm);
\draw[fill] (2,2) circle (0.6mm);
\draw[fill] (3,2) circle (0.6mm);

\draw (0,0)--(3,0); 
\draw (0,1)--(3,1);
\draw (0,1)--(1,2);
\draw (0,0)--(2,2);
\draw (1,0)--(3,2);
\draw (2,0)--(3,1);

\node at (0.3,1.7) {$a_{0,1}$};
\node at (1.3,1.7) {$a_{1,1}$};
\node at (2.3,1.7) {$a_{2,1}$};
\node at (0.1,0.5) {$a_{0,0}$};
\node at (1.1,0.5) {$a_{1,0}$};
\node at (2.1,0.5) {$a_{2,0}$};
\node at (0.6,1.15) {$b_{0,1}$};
\node at (1.6,1.15) {$b_{1,1}$};
\node at (2.6,1.15) {$b_{2,1}$};
\node at (0.6,-0.2) {$b_{0,0}$};
\node at (1.6,-0.2) {$b_{1,0}$};
\node at (2.6,-0.2) {$b_{2,0}$};
\end{tikzpicture}
\end{center}
\vspace{-0.5cm}\caption{\label{fig:periodicity}Left: the graph $\mathcal{G}_{H}$. The dashed lines emphasize the $2 \times 3$ periodicity, but are not part of $\mathcal{G}_{H}$. Right: the assignment of the weights on a $2 \times 3$ block of $\mathcal{G}_{H}$.}
\end{figure}
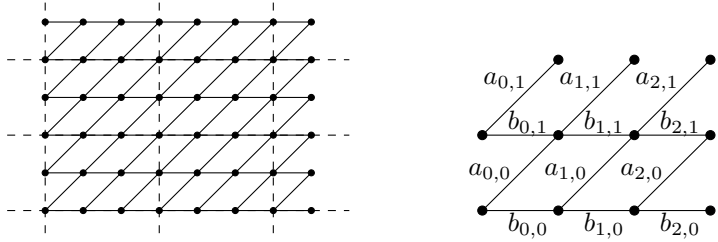

\noindent We define $q$ matrices $A_{0}, \ldots, A_{q-1}$, each of them of size $r \times r$, by 
\begin{align}\label{matrix Al}
& A_{\ell}(z) = \begin{pmatrix}
b_{\ell,0} & a_{\ell,0} & 0 & 0 & \cdots & 0 & 0 & 0  \\
0 & b_{\ell,1} & a_{\ell,1} & 0  & \cdots & 0 & 0 & 0 \\
\vdots & \vdots & \vdots & \vdots & \ddots & \vdots & \vdots & \vdots \\
0 & 0 & 0 & 0 & \cdots & 0 & b_{\ell,r-2} & a_{\ell,r-2} \\
z a_{\ell,r-1} & 0 & 0 & 0 & \cdots & 0 & 0 & b_{\ell,r-1}
\end{pmatrix}, & & \ell = 0,\ldots,q-1,
\end{align}
and for all $m_{1} \in \mathbb{Z}$, we define $A_{\ell + m_{1} q} = A_{\ell}$. We can retrieve the weighting from these matrices by 
\begin{align}\label{integral formula}
\mbox{weight of } \big( (\ell,ry_{1}+j-1),(\ell+1,ry_{2}+k-1) \big) = \frac{1}{2\pi i}\oint_{\gamma} (A_{\ell}(z))_{j,k} z^{y_{1}-y_{2}} \frac{dz}{z},
\end{align}
where $\ell,y_{1},y_{2} \in \mathbb{Z}$, $1 \leq j,k \leq r$,  and $\gamma$ is the unit circle oriented positively. Given three positive integers $L$, $N$ and $M$, we consider the subgraph $\widehat{\mathcal{G}}_{H}$ that consists of the vertices and edges of $\mathcal{G}_{H}$ that lie entirely in the hexagon
\begin{align*}
\mathcal{H} := \{(x,y) \in \mathbb{R}^{2} : 0 \leq x \leq L, \; -\tfrac{1}{2} \leq y \leq N+M-\tfrac{1}{2}, \; -(L-M)-\tfrac{1}{2} \leq y-x \leq  N - \tfrac{1}{2} \},
\end{align*}
see also Figure \ref{fig:hexagon} (left). We say that $\mathfrak{p} : \{0,1,\ldots,L\} \to \mathbb{Z}^{2} \cap \mathcal{H}$ is a path of $\widehat{\mathcal{G}}_{H}$ if $\mathfrak{p}(\ell+1)-\mathfrak{p}(\ell) \in \{0,1\}$ for each $\ell \in \{0,\ldots,L-1\}$. The set of all systems of $N$ non-intersecting paths of $\widehat{\mathcal{G}}_{H}$ is in bijection with the set of all lozenge tilings of $\mathcal{H}$. This one-to-one correspondence is well-known, and is illustrated in Figure \ref{fig:hexagon} (middle and right) for a particular example. The weighting on the edges of $\widehat{\mathcal{G}}_{H}$ naturally induces a weighting on the paths $\{\mathfrak{p}\}$, and on the tilings $\{\mathcal{T}\}$, by
\begin{align*}
& \mbox{weight of } \mathfrak{p} = \prod_{\ell=0}^{L-1} \mbox{weight of } \big( (\ell,\mathfrak{p}(\ell)),(\ell+1,\mathfrak{p}(\ell+1)) \big) \quad \mbox{ and } \quad  \mbox{weight of } \mathcal{T} = \prod_{\mathfrak{p} \in \mathcal{T}} \mbox{weight of } \mathfrak{p}.
\end{align*}
Since the weights on the edges are positive, this defines a probability measure over the set $\{\mathcal{T}\}$ by 
\begin{align}\label{prob measure over the tilings}
\mathbb{P}(\mathcal{T}) = \frac{\mbox{weight of } \mathcal{T}}{\sum_{\mathcal{T}'} \mbox{weight of } \mathcal{T}'},
\end{align}
where the sum is taken over all tilings. By placing points on the paths as shown in Figure \ref{fig:hexagon} (middle), each tiling $\mathcal{T}$ generates a point configuration. Hence, the probability measure \eqref{prob measure over the tilings} can be seen as a discrete point process. As it turns out \cite{GV,L,EM}, this point process is \textit{determinantal}, which implies that all the information about \eqref{prob measure over the tilings} is encoded in a \textit{correlation kernel} $K : \mathbb{Z}^{2} \times \mathbb{Z}^{2} \to \mathbb{R}$. Let $k \geq 1$ be an integer. The determinantal property of the process means that, for any integers $x_{1},\ldots,x_{k},y_{1},\ldots,y_{k}$ satisfying $0 \leq x_{1},\ldots,x_{k} \leq L$ and $(x_{j},y_{j}) \neq (x_{\ell},y_{\ell})$ if $j \neq \ell$, we have
\begin{align}\label{def of determinantal}
\mathbb{P}\bigg[ \begin{matrix}
\mbox{$N$ non-intersecting paths go through } \\
\mbox{each of the points } (x_{1},y_{1}),\ldots,(x_{k},y_{k})
\end{matrix} \bigg] = \det \big[ K(x_{i},y_{i},x_{j},y_{j}) \big]_{i,j=1}^{k}.
\end{align}

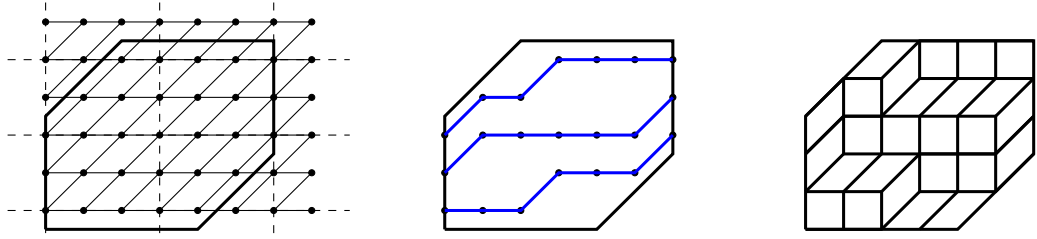
\begin{figure}
\begin{center}
\begin{tikzpicture}
\node at (0,0) {};

\draw[fill] (0,0) circle (0.4mm);
\draw[fill] (0.5,0) circle (0.4mm);
\draw[fill] (1,0) circle (0.4mm);
\draw[fill] (1.5,0) circle (0.4mm);
\draw[fill] (2,0) circle (0.4mm);
\draw[fill] (2.5,0) circle (0.4mm);
\draw[fill] (3,0) circle (0.4mm);
\draw[fill] (3.5,0) circle (0.4mm);

\draw[fill] (0,0.5) circle (0.4mm);
\draw[fill] (0.5,0.5) circle (0.4mm);
\draw[fill] (1,0.5) circle (0.4mm);
\draw[fill] (1.5,0.5) circle (0.4mm);
\draw[fill] (2,0.5) circle (0.4mm);
\draw[fill] (2.5,0.5) circle (0.4mm);
\draw[fill] (3,0.5) circle (0.4mm);
\draw[fill] (3.5,0.5) circle (0.4mm);

\draw[fill] (0,1) circle (0.4mm);
\draw[fill] (0.5,1) circle (0.4mm);
\draw[fill] (1,1) circle (0.4mm);
\draw[fill] (1.5,1) circle (0.4mm);
\draw[fill] (2,1) circle (0.4mm);
\draw[fill] (2.5,1) circle (0.4mm);
\draw[fill] (3,1) circle (0.4mm);
\draw[fill] (3.5,1) circle (0.4mm);

\draw[fill] (0,1.5) circle (0.4mm);
\draw[fill] (0.5,1.5) circle (0.4mm);
\draw[fill] (1,1.5) circle (0.4mm);
\draw[fill] (1.5,1.5) circle (0.4mm);
\draw[fill] (2,1.5) circle (0.4mm);
\draw[fill] (2.5,1.5) circle (0.4mm);
\draw[fill] (3,1.5) circle (0.4mm);
\draw[fill] (3.5,1.5) circle (0.4mm);

\draw[fill] (0,2) circle (0.4mm);
\draw[fill] (0.5,2) circle (0.4mm);
\draw[fill] (1,2) circle (0.4mm);
\draw[fill] (1.5,2) circle (0.4mm);
\draw[fill] (2,2) circle (0.4mm);
\draw[fill] (2.5,2) circle (0.4mm);
\draw[fill] (3,2) circle (0.4mm);
\draw[fill] (3.5,2) circle (0.4mm);

\draw[fill] (0,2.5) circle (0.4mm);
\draw[fill] (0.5,2.5) circle (0.4mm);
\draw[fill] (1,2.5) circle (0.4mm);
\draw[fill] (1.5,2.5) circle (0.4mm);
\draw[fill] (2,2.5) circle (0.4mm);
\draw[fill] (2.5,2.5) circle (0.4mm);
\draw[fill] (3,2.5) circle (0.4mm);
\draw[fill] (3.5,2.5) circle (0.4mm);


\draw (0,0)--(3.5,0);
\draw (0,0.5)--(3.5,0.5);
\draw (0,1)--(3.5,1);
\draw (0,1.5)--(3.5,1.5);
\draw (0,2)--(3.5,2);
\draw (0,2.5)--(3.5,2.5);


\draw (0,0)--(2.5,2.5);

\draw (0.5,0)--(3,2.5);
\draw (1,0)--(3.5,2.5);
\draw (1.5,0)--(3.5,2);
\draw (2,0)--(3.5,1.5);
\draw (2.5,0)--(3.5,1);
\draw (3,0)--(3.5,0.5);

\draw (0,0.5)--(2,2.5);
\draw (0,1)--(1.5,2.5);
\draw (0,1.5)--(1,2.5);
\draw (0,2)--(0.5,2.5);
\draw (0,2.5)--(0,2.5);

\draw[dashed] (-0.5,0)--(4,0);
\draw[dashed] (-0.5,1)--(4,1);
\draw[dashed] (-0.5,2)--(4,2);
\draw[dashed] (0,-0.3)--(0,2.8);
\draw[dashed] (1.5,-0.3)--(1.5,2.8);
\draw[dashed] (3,-0.3)--(3,2.8);

\draw[line width = 0.4mm] (0,-0.25)--++(0,1.5)--++(1,1)--++(2,0)--++(0,-1.5)--++(-1,-1)--++(-2,0);

\end{tikzpicture} \hspace{0.5cm} \begin{tikzpicture}[slave]
\node at (0,0) {};

\draw[fill] (0,0) circle (0.4mm);
\draw[fill] (0.5,0) circle (0.4mm);
\draw[fill] (1,0) circle (0.4mm);

\draw[fill] (0,0.5) circle (0.4mm);
\draw[fill] (1.5,0.5) circle (0.4mm);
\draw[fill] (2,0.5) circle (0.4mm);
\draw[fill] (2.5,0.5) circle (0.4mm);

\draw[fill] (0,1) circle (0.4mm);
\draw[fill] (0.5,1) circle (0.4mm);
\draw[fill] (1,1) circle (0.4mm);
\draw[fill] (1.5,1) circle (0.4mm);
\draw[fill] (2,1) circle (0.4mm);
\draw[fill] (2.5,1) circle (0.4mm);
\draw[fill] (3,1) circle (0.4mm);

\draw[fill] (0.5,1.5) circle (0.4mm);
\draw[fill] (1,1.5) circle (0.4mm);
\draw[fill] (3,1.5) circle (0.4mm);

\draw[fill] (1.5,2) circle (0.4mm);
\draw[fill] (2,2) circle (0.4mm);
\draw[fill] (2.5,2) circle (0.4mm);
\draw[fill] (3,2) circle (0.4mm);

\draw[line width = 0.4mm] (0,-0.25)--++(0,1.5)--++(1,1)--++(2,0)--++(0,-1.5)--++(-1,-1)--++(-2,0);

\draw[line width = 0.4mm,blue] (0,1)--++(0.5,0.5)--++(0.5,0)--++(0.5,0.5)--++(1.5,0);
\draw[line width = 0.4mm,blue] (0,0.5)--++(0.5,0.5)--++(2,0)--++(0.5,0.5);
\draw[line width = 0.4mm,blue] (0,0)--++(1,0)--++(0.5,0.5)--++(1,0)--++(0.5,0.5);
\end{tikzpicture} \hspace{0cm} \begin{tikzpicture}[slave]
\node at (0,0) {};

\draw[line width = 0.4mm] (0,-0.25)--++(0,1.5)--++(1,1)--++(2,0)--++(0,-1.5)--++(-1,-1)--++(-2,0);

\draw[line width = 0.4mm] (0,1.25)--++(0.5,0.5)--++(0.5,0)--++(0.5,0.5)--++(1.5,0);
\draw[line width = 0.4mm] (0,0.75)--++(0.5,0.5)--++(0.5,0)--++(0.5,0.5)--++(1.5,0);

\draw[line width = 0.4mm] (0,0.75)--++(0.5,0.5)--++(2,0)--++(0.5,0.5);
\draw[line width = 0.4mm] (0,0.25)--++(0.5,0.5)--++(2,0)--++(0.5,0.5);

\draw[line width = 0.4mm] (0,0.25)--++(1,0)--++(0.5,0.5)--++(1,0)--++(0.5,0.5);
\draw[line width = 0.4mm] (0,-0.25)--++(1,0)--++(0.5,0.5)--++(1,0)--++(0.5,0.5);

\draw[line width = 0.4mm] (0.5,-0.25)--(0.5,0.25);
\draw[line width = 0.4mm] (0.5,0.75)--(0.5,1.75);

\draw[line width = 0.4mm] (1,-0.25)--(1,0.25);
\draw[line width = 0.4mm] (1,0.75)--(1,1.75);

\draw[line width = 0.4mm] (1.5,0.25)--(1.5,1.25);
\draw[line width = 0.4mm] (1.5,1.75)--(1.5,2.25);

\draw[line width = 0.4mm] (2,0.25)--(2,1.25);
\draw[line width = 0.4mm] (2,1.75)--(2,2.25);

\draw[line width = 0.4mm] (2.5,0.25)--(2.5,1.25);
\draw[line width = 0.4mm] (2.5,1.75)--(2.5,2.25);

\draw[line width = 0.4mm] (0.5,0.25)--(1,0.75);
\draw[line width = 0.4mm] (1.5,-0.25)--(2,0.25);
\draw[line width = 0.4mm] (1.5,1.25)--(2,1.75);
\draw[line width = 0.4mm] (2,1.25)--(2.5,1.75);
\end{tikzpicture}
\caption{\label{fig:hexagon}Left: a hexagon with $2 \times 3$ periodic weightings, and $N=3$, $M=2$, and $L = 6$. Middle: a system of $N$ non-intersecting paths, and the associated points. Right: The corresponding lozenge tiling of the hexagon.}
\end{center}
\end{figure}

\noindent The main result of \cite{DK} is the following integral formula for $K$ involving a matrix valued CD kernel.\footnote{$R_{N}$ in \cite{DK} corresponds to $2\pi i\mathcal{R}_{N}^{W}$ in this paper.}

\begin{theorem}\label{thm: DK} \cite[Theorem 4.7]{DK}. Assume that $L$ is a multiple of $q$, and that $N$ and $M$ are multiples of $r$. For $x_{1},x_{2}\in \{1,\ldots,L-1\}$ and $y_{1},y_{2} \in \mathbb{Z}$, we have
\begin{multline}\label{DK formula}
[K(x_{1},ry_{1}+j,x_{2},ry_{2}+i)]_{i,j=0}^{r-1} = - \frac{\chi_{x_{1}>x_{2}}}{2\pi i} \oint_{\gamma} B_{4}(z)A(z)^{L_{3}}B_{3}(z) dz \\
+ \frac{1}{2\pi i}\oint_{\gamma} \oint_{\gamma} B_{2}(w) A(w)^{L_{2}} \mathcal{R}_{\frac{N}{r}}^{W}(w,z) A(z)^{L_{1}}B_{1}(z) dzdw,
\end{multline}
where $\gamma$ is the unit circle oriented positively,
\begin{align*}
L_{1}=\lfloor \tfrac{x_{1}}{q} \rfloor, \qquad L_{2} = \tfrac{L}{q}-\lceil \tfrac{x_{2}}{q} \rceil, \qquad L_{3} = \max \big\{\lfloor \tfrac{x_{1}}{q} \rfloor - \lceil \tfrac{x_{2}}{q} \rceil ,0 \big\}, \qquad \chi_{x_{1}>x_{2}} = \begin{cases} 
1, & \mbox{if } x_{1}>x_{2}, \\
0, & \mbox{otherwise},
\end{cases}
\end{align*} 
the matrix weight is given by
\begin{align}\label{def of W in tiling model DK}
W(z) = z^{-\frac{M+N}{r}}A(z)^{\frac{L}{q}}, \qquad A(z) = \prod_{\ell = 0}^{q-1}A_{\ell}(z),
\end{align}
and the matrices $B_{1},B_{2},B_{3},B_{4}$ are given by
\begin{align*}
& B_{1}(z) =  z^{-y_{1}-1}\prod_{\ell=q \lfloor \frac{x_{1}}{q} \rfloor}^{x_{1}-1} A_{\ell}(z), & & B_{2}(z) = \frac{z^{y_{2}}}{z^{\frac{M+N}{r}}} \prod_{\ell=x_{2}}^{q \lceil \frac{x_{2}}{q} \rceil -1} A_{\ell}(z), \\
& B_{3}(z) =z^{-y_{1}-1} \begin{cases}
\ds \prod_{\ell=q \lfloor \frac{x_{1}}{q} \rfloor}^{x_{1}-1} A_{\ell}(z), & \mbox{if } \lfloor \frac{x_{1}}{q} \rfloor \geq \lceil \frac{x_{2}}{q} \rceil, \\
I_{r}, & \mbox{otherwise},
\end{cases} & & B_{4}(z) = z^{y_{2}} \begin{cases}
\ds \prod_{\ell=x_{2}}^{q \lceil \frac{x_{2}}{q} \rceil -1} A_{\ell}(z), & \mbox{if } \lfloor \frac{x_{1}}{q} \rfloor \geq \lceil \frac{x_{2}}{q} \rceil, \\
\ds \prod_{\ell=x_{2}}^{x_{1}} A_{\ell}(z), & \mbox{otherwise}.
\end{cases}
\end{align*}
\end{theorem}
\begin{remark}
In this set-up, the existence of $\mathcal{R}_{N/r}^{W}$ is ensured from \cite[Lemma 4.8]{DK}.
\end{remark}
If one desires to understand the fine asymptotic structure of the lozenge tiling model under consideration, one usually needs to obtain asymptotics for $K(x_{1},y_{1},x_{2},y_{2})$ as $L, M, N \to + \infty$, and simultaneously $x_{1},y_{1},x_{2},y_{2} \to + \infty$ at certain critical speeds. Since the matrix products that appear in $B_{1},\ldots,B_{4}$ always involve at most $q-1$ matrices, the $B_{j}$'s are not an obstacle to an asymptotic analysis. Slightly more problematic is the fact that $L_{1},L_{2},L_{3} \to + \infty$ in this regime, and that these quantities are exponents of $A$ in \eqref{DK formula}. But this, in itself, is also not a serious obstacle if $A$ is diagonalizable. Since $W$ depends on the large parameters $L, M$ and $N$, see \eqref{def of W in tiling model DK}, the real challenge is to obtain asymptotics of $\mathcal{R}_{N/r}^{W}$, especially when $W$ is matrix-valued. Theorems \ref{thm: reproducing kernel Riemann surface} and \ref{thm: reprod Rcal U} allows to circumvent (or simplify) this serious technical obstacle, provided $A$ satisfies the following assumption. 
\begin{assumption}\label{ass: matrix A}
The rational $r \times r$ matrix valued function $z \mapsto A(z)$ has no pole on $\gamma$ and is diagonalizable for all but finitely many $z \in \mathbb{C}$.
\end{assumption}
Note that, if $W$ is of the form $W(z) = z^{-\frac{M+N}{r}}A(z)^{\frac{L}{q}}$ and if Assumption \ref{ass: matrix A} holds, then automatically Assumption \ref{ass: weight diag} holds as well.

Since $A$ is rational, its eigenvalues $\widehat{\lambda}_{1}(z),\ldots,\widehat{\lambda}_{r}(z)$ are (branches of) meromorphic functions on $\mathbb{C}$, which together define a meromorphic function $\widehat{\lambda}$ on an $r$-sheeted compect Riemann surface $\mathcal{M}$. For definiteness, we choose the numbering of the sheets such that 
\begin{align}\label{def of lambda hat}
& \z \mapsto \widehat{\lambda}(\z) = \widehat{\lambda}_{k}(z), & & \mbox{if $\z$ is on the $k$-th sheet of $\mathcal{M}$}.
\end{align}
Assumption \ref{ass: matrix A} implies that for all but finitely many $z \in \mathbb{C}$, we can write 
\begin{align}\label{diagonal form of A}
A(z) =  E(z) \widehat{\Lambda}(z) E(z)^{-1}, \qquad \widehat{\Lambda}(z) = \diag(\widehat{\lambda}_{1}(z),\ldots,\widehat{\lambda}_{r}(z)),
\end{align}
and as in Section \ref{section: introduction}, we can (and do) choose the matrix of eigenvectors $E$ such that the functions $\e$ and $\widehat{\e}$ defined in \eqref{def of lambda RS}--\eqref{def of efrak inv RS} are meromorphic on $\mathcal{M}$. Since
\begin{align*}
W(z) = z^{-\frac{M+N}{r}}A(z)^{\frac{L}{q}}, \qquad \mbox{we have} \qquad \lambda_{k}(z)=z^{-\frac{M+N}{r}}\widehat{\lambda}_{k}(z)^{\frac{L}{q}}, \quad k=1,\ldots,r,
\end{align*}
and therefore the function $\lambda$ defined in \eqref{def of lambda RS} is also meromorphic on $\mathcal{M}$.

\medskip Theorem \ref{thm: double contour integrals simplifications} below is our general result for tiling models with doubly periodic weightings. If $\mathcal{M}$ is a connected Riemann surface of genus $0$, then formula \eqref{I simplified in thm} provides a new double contour formula representation of $K$. The main advantage of this formula is that all the quantities involving the (potentially large) parameters $L$, $M$ and $N$ are now scalar. If $\mathcal{M}$ has either no genus (this is the case if $\mathcal{M}$ is not a connected Riemann surface), or if $\mathcal{M}$ has genus $1$ or more, then formula \eqref{I simplified in thm} does not apply. In this case we instead give a new representation of the kernel which involve a scalar kernel in a Riemann surface, see formula \eqref{simplif I on Riemann surface} below. Here too, all the quantities involving the parameters $L$, $M$ and $N$ are scalar. Therefore, we expect that this formula also leads to a simpler asymptotic analysis than \eqref{DK formula}, although we admit that this still remains to be shown in a concrete situation. 

\begin{theorem}\label{thm: double contour integrals simplifications}
Suppose that $A$ satisfies Assumption \ref{ass: matrix A}. Let $\mathcal{M}$ be an $r$-sheeted Riemann surface such that the functions $\widehat{\lambda}$, $\e$ and $\widehat{\e}$ defined in  \eqref{def of lambda hat}, \eqref{def of efrak RS}, \eqref{def of efrak inv RS} are meromorphic, and define $R_{N/r}^{\lambda}$ as in \eqref{def of RM}. The formula \eqref{DK formula} can be simplified as follows:
\begin{align}
[K(x_{1},ry_{1}+j,x_{2},ry_{2}+i)]_{i,j=0}^{r-1} = & \; \frac{1}{2\pi i} \oint_{\gamma_{\mathcal{M}}}\oint_{\gamma_{\mathcal{M}}} \widehat{\lambda}(\w)^{L_{2}} \widehat{\lambda}(\z)^{L_{1}} B_{2}(w)\e(\w)  R_{N/r}^{\lambda}(\w,\z) \widehat{\e}(\z) B_{1}(z) dz dw \nonumber \\
& - \frac{\chi}{2\pi i} \oint_{\gamma_{\mathcal{M}}} \widehat{\lambda}(\z)^{L_{3}} B_{4}(z)\e(\z) \widehat{\e}(\z)B_{3}(z) dz. \label{simplif I on Riemann surface}
\end{align}
If furthermore $\mathcal{M}$ is of genus $0$, then the right-hand-side of \eqref{simplif I on Riemann surface} can be further simplified into
\begin{align}
& \frac{1}{2\pi i} \oint_{\gamma_{\mathbb{C}}}\oint_{\gamma_{\mathbb{C}}} \frac{\mathfrak{R}_{N}^{\mathcal{W}}(\omega,\zeta)}{\widehat{h}(\omega)h(\zeta)} \widehat{\lambda}(\varphi(\omega))^{L_{2}} \widehat{\lambda}(\varphi(\zeta))^{L_{1}} \phi'(\omega)\phi'(\zeta) B_{2}(\phi(\omega))\e(\varphi(\omega))  \widehat{\e}(\varphi(\zeta)) B_{1}(\phi(\zeta)) d\zeta d\omega \nonumber \\
& - \frac{\chi}{2\pi i} \oint_{\gamma_{\mathbb{C}}} \widehat{\lambda}(\varphi(\zeta))^{L_{3}} B_{4}(\phi(\zeta))\e(\varphi(\zeta))\widehat{\e}(\varphi(\zeta))B_{3}(\phi(\zeta)) \phi'(\zeta)d\zeta, \label{I simplified in thm}
\end{align}
where $\mathfrak{R}_{N}^{\mathcal{W}}$, $h$, $\widehat{h}$, $\varphi$, $\phi$, $\mathcal{W}$, $\gamma_{\mathbb{C}}$ are defined in Definitions \ref{def: varphi}--\ref{def: scalar weight}.
\end{theorem}
\begin{proof}
Substituting $A(z) = E(z) \widehat{\Lambda}(z) E(z)^{-1}$ in \eqref{DK formula}, we get 
\begin{align*}
& [K(x_{1},ry_{1}+j,x_{2},ry_{2}+i)]_{i,j=0}^{r-1} = - \frac{\chi}{2\pi i} \oint_{\gamma} B_{4}(z)E(z)\widehat{\Lambda}(z)^{L_{3}} E(z)^{-1}B_{3}(z) dz \\
& +\frac{1}{2\pi i}\oint_{\gamma}\oint_{\gamma}B_{2}(w)E(w) \widehat{\Lambda}(w)^{L_{2}}E(w)^{-1} \mathcal{R}_{N/r}^{W}(w,z)E(z)\widehat{\Lambda}(z)^{L_{1}} E(z)^{-1} B_{1}(z) dzdw.
\end{align*}
Note that for any integer $x$, we can write 
\begin{align*}
\widehat{\Lambda}(z)^{x} = \sum_{j=1}^{r} e_{j}e_{j}^{T} \widehat{\lambda}_{j}(z)^{x} = \sum_{j=1}^{r} e_{j}e_{j}^{T} \widehat{\lambda}(z^{(j)})^{x}.
\end{align*}
Hence, using \eqref{def of RM}, we obtain
\begin{align*}
& [K(x_{1},ry_{1}+j,x_{2},ry_{2}+i)]_{i,j=0}^{r-1} = - \frac{\chi}{2\pi i} \sum_{j=1}^{r} \oint_{\gamma} B_{4}(z)\e(z^{(j)})\widehat{\lambda}(z^{(j)})^{L_{3}} \widehat{\e}(z^{(j)})B_{3}(z) dz  \;  \\
& +\frac{1}{2\pi i} \sum_{j=1}^{r}\sum_{k=1}^{r} \oint_{\gamma}\oint_{\gamma}B_{2}(w)\e(w^{(k)})\widehat{\lambda}(w^{(k)})^{L_{2}}R_{N/r}^{\lambda}(w^{(k)},z^{(j)})\widehat{\lambda}(z^{(j)})^{L_{1}} \widehat{\e}(z^{(j)}) B_{1}(z) dzdw,
\end{align*}
which is equivalent to \eqref{simplif I on Riemann surface} (recal that $\gamma_{\mathcal{M}} = \cup_{j=1}^{r} \gamma^{(j)}$). The formula \eqref{I simplified in thm} then directly follows from \eqref{def of RU}, the change of variables $\z = \varphi(\zeta)$ and $\w = \varphi(\omega)$, and the definition $\gamma_{\mathbb{C}} = \varphi^{-1}(\gamma_{\mathcal{M}})$.
\end{proof}



In the rest of this section, we give three applications of Theorem \ref{thm: double contour integrals simplifications}. 


\subsection{The uniform measure viewed as an $r \times 1$ periodic weighting}\label{subsection: uniform measure}
The uniform measure over the lozenge tilings of the hexagon has already been extensively studied \cite{Jptrf,BKMM,Gorin,Aggarwal}. It corresponds to the case where all edges of $\mathcal{G}_{H}$ are assigned the same positive number, say $1$. In the terminology introduced at the beginning of this section, this is a model with a $1 \times 1$ periodic weighting. Therefore, the correlation kernel of this model is given by Theorem \ref{thm: DK} with $r=q=1$ and $A(z)=A_{0}(z) = 1+z$:
\begin{align}
K(x_{1},y_{1},x_{2},y_{2}) = & \frac{1}{2\pi i} \oint_{\gamma}\oint_{\gamma} \frac{(1+w)^{L-x_{2}}}{w^{M+N-y_{2}}}\mathcal{R}_{N}^{\widetilde{W}}(w,z) \frac{(1+z)^{x_{1}}}{z^{y_{1}+1}}dzdw \nonumber \\
& - \frac{\chi_{x_{1}>x_{2}}}{2\pi i}\oint_{\gamma} (1+z)^{x_{1}-x_{2}}z^{y_{2}-y_{1}-1}dz, \label{unif kernel scalar form}
\end{align}
where $\widetilde{W}(z) = (1+z)^{L}z^{-M-N}$ and $\gamma$ is the unit circle. In this subsection, we do not provide new results, but we provide an example which we believe illustrates well how to use Theorem \ref{thm: double contour integrals simplifications}. Since the uniform measure is $1 \times 1$ periodic, it can in particular be viewed as a model with an $r \times 1$ periodic weighting, $r \geq 2$. Assume that $M$ and $N$ are multiple of $r$. By \eqref{matrix Al}, the transition matrix $A(z)=A_{0}(z)$ is given by
\begin{align}\label{A rx1}
A(z) = \begin{pmatrix}
1 & 1 & 0 & 0 & \cdots & 0 & 0 \\
0 & 1 & 1 & 0 & \cdots & 0 & 0 \\
\vdots & \vdots & \vdots & \vdots & \ddots & \vdots & \vdots \\
0 & 0 & 0 & 0 & \cdots & 1 & 1 \\
z & 0 & 0 & 0 & \cdots & 0 & 1
\end{pmatrix}.
\end{align}
Applying Theorem \ref{thm: DK} to this case, we obtain that for any $x_{1},x_{2} \in \{1,\ldots,L-1\}$ and $y_{1},y_{2} \in \mathbb{Z}$,
\begin{align*}
[K(x_{1},ry_{1}+j,x_{2},ry_{2}+i)]_{i,j=0}^{r-1} = & \, \frac{1}{2\pi i}\oint_{\gamma} \oint_{\gamma}  \frac{A(w)^{L-x_{2}}}{w^{\frac{M+N}{r}-y_{2}}} \mathcal{R}_{N/r}^{W}(w,z) \frac{A(z)^{x_{1}}}{z^{y_{1}+1}} dzdw  \\
&  - \frac{\chi_{x_{1}>x_{2}}}{2\pi i} \oint_{\gamma} A(z)^{x_{1}-x_{2}}z^{y_{2}-y_{1}-1} dz, \qquad\qquad 
\end{align*}
or equivalently, 
\begin{align}
K(x_{1},y_{1},x_{2},y_{2}) = & \frac{1}{2\pi i}\oint_{\gamma}\oint_{\gamma} e_{y_{2}-r\lfloor \frac{y_{2}}{r} \rfloor +1}^{T}\frac{A(w)^{L-x_{2}}}{w^{\frac{M+N}{r}- \lfloor \frac{y_{2}}{r} \rfloor }} \mathcal{R}_{N/r}^{W}(w,z) \frac{A(z)^{x_{1}}}{z^{\lfloor \frac{y_{1}}{r} \rfloor +1}}e_{y_{1}-r\lfloor \frac{y_{1}}{r} \rfloor +1} dzdw \nonumber \\
& - \frac{\chi_{x_{1}>x_{2}}}{2\pi i}\oint_{\gamma} e_{y_{2}-r\lfloor \frac{y_{2}}{r} \rfloor +1}^{T} A(z)^{x_{1}-x_{2}} e_{y_{1}-r \lfloor \frac{y_{1}}{r} \rfloor +1} z^{\lfloor \frac{y_{2}}{r} \rfloor-\lfloor \frac{y_{1}}{r} \rfloor-1} dz, \label{unif kernel matrix form}
\end{align}
where $W(z) = z^{-\frac{M+N}{r}}A(z)^{L}$ and we recall that $e_{j}$ denotes the $j$-th column of $I_{r}$. Note that $W$ is exactly the matrix weight written in \eqref{W example rxr} (see also Example \ref{example: scalar 2}) after replacing $R$ by $\frac{M+N}{r}$. Since the right-hand sides of \eqref{unif kernel scalar form} and \eqref{unif kernel matrix form} are two correlation kernels of the same determinantal process, they must obviously be related (though a priori not necessarily equal\footnote{The correlation kernel of a given point process is not unique. For example, the kernels $K$ and $\tilde{K}(x_{1},y_{1},x_{2},y_{2}):= K(x_{1},y_{1},x_{2},y_{2}) \frac{f_{1}(x_{1})}{f_{1}(x_{2})}\frac{f_{2}(y_{1})}{f_{2}(y_{2})}$, where $f_{1}$ and $f_{2}$ are arbitrary non-vanishing functions, define the same point process, because the determinants in \eqref{def of determinantal} remain unchanged.}). However, this is fairly non-trivial task to verify this directly from the formulas. Here, we provide a direct proof for this fact using Theorem \ref{thm: double contour integrals simplifications}. 
\begin{proposition}
The right-hand sides of \eqref{unif kernel scalar form} and \eqref{unif kernel matrix form} are equal.
\end{proposition}
\begin{proof} 
This proof relies on the computation done in Example \ref{example: scalar 2}. Since $\mathcal{M}$ is of genus $0$, we infer from Theorem \ref{thm: double contour integrals simplifications} that the right-hand side of \eqref{unif kernel matrix form} can be rewritten as
\begin{align*}
& \frac{1}{2\pi i} \oint_{\gamma}\oint_{\gamma} \frac{\mathcal{R}_{N}^{\mathcal{W}}(\omega,\zeta)}{\widehat{h}(\omega)h(\zeta)} \frac{\widehat{\lambda}(\varphi(\omega))^{L-x_{2}}}{\phi(\omega)^{\frac{M+N}{r}- \lfloor \frac{y_{2}}{r} \rfloor }}   \frac{\widehat{\lambda}(\varphi(\zeta))^{x_{1}}}{\phi(\zeta)^{\lfloor \frac{y_{1}}{r} \rfloor +1 }} e_{y_{2}-r\lfloor \frac{y_{2}}{r} \rfloor +1}^{T}\e(\varphi(\omega))  \widehat{\e}(\varphi(\zeta)) e_{y_{1}-r\lfloor \frac{y_{1}}{r} \rfloor +1} \phi'(\omega)\phi'(\zeta)d\zeta d\omega \\
& - \frac{\chi_{x_{1}>x_{2}}}{2\pi i}\oint_{\gamma} \widehat{\lambda}(\varphi(\zeta))^{x_{1}-x_{2}} e_{y_{2}-r\lfloor \frac{y_{2}}{r} \rfloor +1}^{T} \e(\varphi(\zeta))\widehat{\e}(\varphi(\zeta))  e_{y_{1}-r\lfloor \frac{y_{1}}{r} \rfloor +1} \phi(\zeta)^{\lfloor \frac{y_{2}}{r} \rfloor-\lfloor \frac{y_{1}}{r} \rfloor-1} \phi'(\zeta) d\zeta,
\end{align*}
where we have also used that $\mathfrak{R}_{N}^{W}=\mathcal{R}_{N}^{\mathcal{W}}$ and $\gamma_{\mathbb{C}} = \gamma$. Using the explicit expressions for $\e$, $\widehat{\e}$, $\lambda$, $\varphi$, $\phi$, $h$ and $\widehat{h}$ from Example \ref{example: scalar 2}, and noting that $\widehat{\lambda}((z,\eta))=1+ \eta$, this gives
\begin{align}
& \frac{1}{2\pi i} \oint_{\gamma}\oint_{\gamma} \frac{\mathcal{R}_{N}^{\mathcal{W}}(\omega,\zeta)}{\omega^{r-1}} \frac{(1+\omega)^{L-x_{2}}}{\omega^{M+N- r\lfloor \frac{y_{2}}{r} \rfloor }}  \frac{(1+\zeta)^{x_{1}}}{\zeta^{r\lfloor \frac{y_{1}}{r} \rfloor +r }} \omega^{y_{2}-r\lfloor \frac{y_{2}}{r} \rfloor}  \zeta^{-y_{1}+r\lfloor \frac{y_{1}}{r} \rfloor} \frac{1}{r} r \omega^{r-1} r \zeta^{r-1} d\zeta d\omega \nonumber \\
& - \frac{\chi_{x_{1}>x_{2}}}{2\pi i}\oint_{\gamma} (1+\zeta)^{x_{1}-x_{2}} \zeta^{y_{2}-r\lfloor \frac{y_{2}}{r} \rfloor} \frac{\zeta^{-y_{1}+r\lfloor \frac{y_{1}}{r} \rfloor }}{r} \zeta^{r\lfloor \frac{y_{2}}{r} \rfloor-r\lfloor \frac{y_{1}}{r} \rfloor-r} r \zeta^{r-1} d\zeta. \label{lol3}
\end{align}
Note that the functions $\widetilde{W}$ and $\mathcal{W}$, defined respectively below \eqref{unif kernel scalar form} and in \eqref{weight of example 1.18}, are related by $\widetilde{W}(z) = \frac{1}{r}\mathcal{W}(z)$. Therefore, it is straightforward to deduce from \eqref{reproducing kernel} that $\mathcal{R}_{N}^{\mathcal{W}}(\omega,\zeta)=\frac{1}{r}\mathcal{R}_{N}^{\widetilde{W}}(\omega,\zeta)$. Substituting this relation into \eqref{lol3} and simplifying, we find \eqref{unif kernel scalar form}.
\end{proof}

\subsection{Lozenge tiling models with a $2 \times 1$ periodic weighting}\label{subsection: 2x1 periodic weightings}

The main result of this subsection is an explicit new double contour integral formula for the kernel of an arbitrary lozenge tiling model with a $2 \times 1$ periodic weighting. In this formula, the integrand only involves a scalar CD kernel. 

\medskip Let us set $r=2$ and $q=1$ in \eqref{weight on a block}--\eqref{periodic weightings}. In the most general setting, the transition matrix $A(z)=A_{0}(z)$ is given by
\begin{align}\label{A for 2x1}
A(z) = \begin{pmatrix}
b_{0} & a_{0} \\ a_{1}z & b_{1}
\end{pmatrix}, \qquad a_{0},a_{1},b_{0},b_{1}>0,
\end{align}
where we have written $a_{j},b_{j}$ for the edge weights instead of $a_{0,j},b_{0,j}$ to lighten the notation. 
 Let us assume that $M$ and $N$ are multiples of $2$. We know from Theorem \ref{thm: DK} that the kernel $K$ associated to \eqref{prob measure over the tilings} is given by 
\begin{multline}\label{DK formula 2x1}
[K(x_{1},ry_{1}+j,x_{2},ry_{2}+i)]_{i,j=0}^{1} = - \frac{\chi_{x_{1}>x_{2}}}{2\pi i} \oint_{\gamma} A(z)^{x_{1}-x_{2}}z^{y_{2}-y_{1}-1} dz \\
+ \frac{1}{2\pi i}\oint_{\gamma} \oint_{\gamma} A(w)^{L-x_{2}} \mathcal{R}_{N/2}^{W}(w,z) A(z)^{x_{1}} \frac{dzdw}{w^{\frac{M+N}{2}-y_{2}}z^{y_{1}+1}},
\end{multline}
for $x_{1},x_{2} \in \{1,...,L-1\}$ and $y_{1},y_{2} \in \mathbb{Z}$, with $W(z) = z^{-\frac{M+N}{2}}A(z)^{L}$. The above formula involves a CD kernel of size $2 \times 2$. Theorem \ref{thm: double contour integrals simplifications} allows to simplify \eqref{DK formula 2x1} as follows.
\begin{theorem}\label{thm: 2x1}
For $x_{1},x_{2} \in \{1,\ldots,L-1\}$ and $y_{1},y_{2} \in \mathbb{Z}$, we have
\begin{multline*}
\hspace{-0.3cm}[K(x_{1},2y_{1}+j,x_{2},2y_{2}+i)]_{i,j=0}^{1} = - \frac{\chi_{x_{1}>x_{2}}}{2\pi i} \oint_{\gamma_{\mathbb{C}}} \frac{\big( \frac{b_{0}+b_{1}+\zeta}{2} \big)^{x_{1} - x_{2}}}{\big(\frac{\zeta^{2}-(b_{0}-b_{1})^{2}}{4a_{0}a_{1}}\big)^{y_{1}-y_{2}+1}} \begin{pmatrix}
\frac{\zeta-b_{1}+b_{0}}{2} & a_{0} \\
\frac{\zeta^{2}-(b_{1}-b_{0})^{2}}{4a_{0}} & \frac{\zeta + b_{1}-b_{0}}{2}
\end{pmatrix}  \frac{d\zeta}{2a_{0}a_{1}} \\
+ \frac{1}{2\pi i} \oint_{\gamma_{\mathbb{C}}}\oint_{\gamma_{\mathbb{C}}}  \frac{\mathcal{R}_{N}^{\mathcal{W}}(\omega,\zeta) \big( \frac{b_{0}+b_{1}+\zeta}{2} \big)^{x_{1}} \left( \frac{b_{0}+b_{1}+\omega}{2} \right)^{L-x_{2}}}{ \big(\frac{\zeta^{2}-(b_{1}-b_{0})^{2}}{4a_{1}a_{0}}\big)^{y_{1}+1} \big(\frac{\omega^{2}-(b_{1}-b_{0})^{2}}{4a_{1}a_{0}}\big)^{\frac{M+N}{2}-y_{2}}} \begin{pmatrix}
\frac{\zeta-(b_{1}-b_{0})}{2} & a_{0} \\
\frac{(\zeta-b_{1}+b_{0})(\omega+b_{1}-b_{0})}{4a_{0}} & \frac{\omega + b_{1} - b_{0}}{2}
\end{pmatrix} \frac{d\zeta  d\omega}{4a_{0}^{2}a_{1}^{2}},
\end{multline*}
where $\mathcal{W}(\omega) = \frac{1}{2a_{0}a_{1}}(\frac{b_{0}+b_{1}+\zeta}{2})^{L} (\frac{4a_{0}a_{1}}{\zeta^{2}-(b_{0}-b_{1})^{2}})^{\frac{M+N}{2}}$. The contour $\gamma_{\mathbb{C}}$ is a circle oriented positively and surrounding both $b_{0}-b_{1}$ and $b_{1}-b_{0}$. 
\end{theorem}
\begin{remark}
Note that none of the (possibly large) parameters $L, M, N, x_{1},y_{1},x_{2},y_{2}$ appear in the remaining matrices in the integrands. Therefore, this formula is expected to lead to a much simpler asymptotic analysis than \eqref{DK formula}, see also the discussion above Theorem \ref{thm: double contour integrals simplifications}. 
\end{remark}
\begin{proof}
It is a simple computation to verify that $A$ in \eqref{A for 2x1} can be written in the form \eqref{diagonal form of A} with
\begin{align}
& E(z) = \begin{pmatrix}
1 & 1 \\
\frac{b_{1}-b_{0}+\sqrt{\Delta}}{2a_{0}} & \frac{b_{1}-b_{0}-\sqrt{\Delta}}{2a_{0}}
\end{pmatrix}, \qquad E(z)^{-1} = \begin{pmatrix}
- \frac{b_{1}-b_{0}-\sqrt{\Delta}}{2\sqrt{\Delta}} & \frac{a_{0}}{\sqrt{\Delta}} \\
- \frac{b_{1}-b_{0}+\sqrt{\Delta}}{-2\sqrt{\Delta}} & \frac{a_{0}}{-\sqrt{\Delta}}
\end{pmatrix}, \label{E and E inv 2x1} \\
& \widehat{\Lambda}(z) = \diag \bigg(
\frac{b_{0}+b_{1}+\sqrt{\Delta}}{2}, \frac{b_{0}+b_{1}-\sqrt{\Delta}}{2}\bigg), \qquad \Delta = \Delta(z) = 4a_{0}a_{1}(z-z_{1}), \quad z_{1} = -\frac{(b_{0}-b_{1})^{2}}{4a_{0}a_{1}}, \nonumber 
\end{align}
where the principal branch is chosen for $\sqrt{\Delta}$. In particular $A$ satisfies Assumption \ref{ass: matrix A}. Let $\mathcal{M}$ be the genus $0$ Riemann surface associated to $\{(z,\eta) \in \mathbb{C}^{2}: \eta^{2} = \Delta(z)\}$. We view $\mathcal{M}$ as two copies of $\widehat{\mathbb{C}}$ that are glued along $(-\infty,z_{1})$, with $\eta = \sqrt{\Delta(z)}$ on the first sheet, and $\eta = -\sqrt{\Delta(z)}$ on the second sheet. The function 
\begin{align*}
\varphi(\zeta) = (\phi(\zeta),\zeta) \qquad \mbox{ with } \qquad \phi(\zeta) = z_{1}+\frac{\zeta^{2}}{4a_{0}a_{1}}
\end{align*}
is a bijection from $\widehat{\mathbb{C}}$ to $\mathcal{M}$ whose inverse is $\varphi^{-1}((z,\eta)) = \eta$. Substituting \eqref{E and E inv 2x1} in \eqref{def of lambda RS}--\eqref{def of efrak inv RS}, we obtain the following expressions
\begin{align}\label{e em1 and lambda hat 2x1}
\e(\varphi(\zeta))^{T} = \begin{pmatrix}
1 & \frac{b_{1}-b_{0}+\zeta}{2a_{0}}
\end{pmatrix}, \qquad \widehat{\e}(\varphi(\zeta)) = \begin{pmatrix}
\frac{\zeta+b_{0}-b_{1}}{2 \zeta} & \frac{a_{0}}{\zeta}
\end{pmatrix}, \qquad \lambda(\varphi(\zeta)) =\frac{\widehat{\lambda}(\varphi(\zeta))^{L}}{\phi(\zeta)^{\frac{M+N}{2}}},
\end{align}
with $\widehat{\lambda}(\varphi(\zeta)) = \frac{b_{1}+b_{0}+\zeta}{2}$. The function $\e$ has no zero and a simple pole at $\infty^{(1)}$, while $\widehat{\e}$ has no zero and a simple pole at $(z_{1},0)=z_{1}^{(1)}=z_{1}^{(2)}$. Hence, from Definition \ref{def: nz and nz hat} we have
\begin{align*}
\mathcal{Z} = \emptyset, \qquad \mathcal{Q} = \{\infty^{(1)}\}, \qquad n_{\infty^{(1)}} = -1, \qquad \widehat{\mathcal{Z}} = \emptyset, \qquad \widehat{\mathcal{Q}} = \{z_{1}^{(1)}\}, \qquad \widehat{n}_{z_{1}^{(1)}} = -1.
\end{align*}
Since $\varphi(\infty) = \infty^{(1)} = \infty^{(2)}$, the functions $h$ and $\widehat{h}$ defined in \eqref{def of h}-\eqref{def of h hat} reduce here to $h(\zeta) = 1$ and $\widehat{h}(\zeta) = \zeta$, and the scalar weight \eqref{scalar weight} is given by 
\begin{align*}
\mathcal{W}(\zeta) = \frac{\lambda(\varphi(\zeta))}{h(\zeta)\widehat{h}(\zeta)} \phi'(\zeta) = \frac{1}{2a_{0}a_{1}}\Big(\frac{b_{0}+b_{1}+\zeta}{2}\Big)^{L} \Big(\frac{4a_{0}a_{1}}{\zeta^{2}-(b_{0}-b_{1})^{2}}\Big)^{\frac{M+N}{2}}.
\end{align*}
Since
\begin{align*}
-\sum_{\z \in \mathcal{Z}\cup \mathcal{Q}} n_{\z} = 1,
\end{align*}
it follows from Theorem \ref{thm: reprod Rcal U} (e) that $\mathfrak{R}_{N}^{\mathcal{W}} = \mathcal{R}_{N}^{\mathcal{W}}$. The simplified formula for $K$ is now obtained from \eqref{I simplified in thm}, \eqref{e em1 and lambda hat 2x1} and a direct computation. It only remains to determine the shape of $\gamma_{\mathbb{C}}$. Note that $\gamma$ in \eqref{DK formula 2x1} can be deformed into any closed curve surrounding $0$, and in particular into a small circle surrounding $0$.  Recalling the definition $\gamma_{\mathbb{C}}= \varphi^{-1}(\gamma_{\mathcal{M}})=\varphi^{-1}( \cup_{j=1}^{2} \gamma^{(j)})$, this shows that $\gamma_{\mathbb{C}}$ can be determined from a local analysis of $\varphi^{-1}(\z)$ around $0^{(1)}=(0,\sqrt{\Delta(0)})$ and $0^{(2)}=(0,-\sqrt{\Delta(0)})$. Let us discuss first the case $b_{0} \neq b_{1}$, for which we have $0^{(1)} \neq 0^{(2)}$. Since $\varphi^{-1}(0^{(1)}) = |b_{0}-b_{1}|$ and $\varphi^{-1}(0^{(2)}) = -|b_{0}-b_{1}|$, it follows that the contour $\gamma_{\mathbb{C}}$ can be chosen as the union of two small circles oriented positively: one circle surrounds $b_{0}-b_{1}$, and the other one surrounds $b_{1}-b_{0}$. On the other hand, if $b_{0} = b_{1}$, then $0^{(1)}=0^{(2)}$ is a branch point of $\mathcal{M}$, and $\varphi^{-1}$ maps $\gamma^{(1)} \cup \gamma^{(2)}$ into a small circle surrounding $0$. We conclude that in all cases, $\gamma$ can be chosen to be a single circle oriented positively and enclosing $\pm(b_{0}-b_{1})$.
\end{proof}
\subsection{Lozenge tiling models with a $2 \times 2$ periodic weighting}\label{subsection: 2x2 periodic weightings} 
In this subsection, we show that the kernel of any lozenge tiling model with a $2 \times 2$ periodic weighting admits a double contour formula representation involving a scalar CD kernel. 

\medskip Lozenge tiling models with a $2 \times 2$ periodic weighting are defined by \eqref{weight on a block}--\eqref{periodic weightings} and \eqref{prob measure over the tilings} with $r=2$ and $q=2$. Assume that $M$, $N$ and $L$ are multiples of $2$. The transition matrices $A_{0}$ and $A_{1}$, defined in \eqref{matrix Al}, are given by
\begin{align*}
& A_{0}(z) = \begin{pmatrix}
b_{0,0} & a_{0,0} \\
a_{0,1} z & b_{0,1}
\end{pmatrix} , & & A_{1}(z) = \begin{pmatrix}
b_{1,0} & a_{1,0} \\
a_{1,1} z & b_{1,1}
\end{pmatrix},
\end{align*}
where $a_{\ell,j},b_{\ell,j} > 0$ for all $\ell, j \in \{0,1\}$. The quantities $\{L_{j}\}_{j=1}^{3}$ and $\{B_{j}\}_{j=1}^{4}$ of Theorem \ref{thm: DK} depend on the parity of $x_{1}$ and $x_{2}$. Therefore, to ease the notation, we now invoke Theorem \ref{thm: DK} with $x_{1}$ and $x_{2}$ replaced by $2x_{1}+\epsilon_{1}$ and $2x_{2}-\epsilon_{2}$, respectively, where $\epsilon_{1},\epsilon_{2} \in \{0,1\}$.

\medskip  For $\epsilon_{1},\epsilon_{2} \in \{0,1\}$, $y_{1},y_{2} \in \mathbb{Z}$, and for integers $x_{1},x_{2}$ such that $2x_{1}+\epsilon_{1},2x_{2}-\epsilon_{2} \in \{1,\ldots,L-1\}$, we have
\begin{multline}\label{DK formula 2x2 1}
[K(2x_{1}+\epsilon_{1},2y_{1}+j,2x_{2}-\epsilon_{2},2y_{2}+i)]_{i,j=0}^{1} = - \frac{\chi_{2x_{1}+\epsilon_{1}>2x_{2}-\epsilon_{2}}}{2\pi i} \oint_{\gamma} A_{1}(z)^{\epsilon_{2}}A(z)^{x_{1}-x_{2}}A_{0}(z)^{\epsilon_{1}} z^{y_{2}-y_{1}-1}dz \\
+ \frac{1}{2\pi i}\oint_{\gamma} \oint_{\gamma} \frac{A_{1}(w)^{\epsilon_{2}}}{w^{\frac{M+N}{2}-y_{2}}} A(w)^{\frac{L}{2}-x_{2}} \mathcal{R}_{N/2}^{W}(w,z) A(z)^{x_{1}}\frac{A_{0}(z)^{\epsilon_{1}}}{z^{y_{1}+1}} dzdw,
\end{multline}
where $\gamma$ is the unit circle oriented positively, and
\begin{align*}
W(z) = z^{-\frac{M+N}{2}}A(z)^{\frac{L}{2}}, \qquad \mbox{ with } \qquad A(z) = A_{0}(z) A_{1}(z).
\end{align*}

\begin{theorem}\label{thm: 2x2}
The right-hand side of \eqref{DK formula 2x2 1} can be simplified to
\begin{align}
& - \frac{\chi_{2x_{1}+\epsilon_{1}>2x_{2}-\epsilon_{2}}}{2\pi i} \oint_{\gamma_{\mathbb{C}}} \widehat{\lambda}(\varphi(\zeta))^{x_{1}-x_{2}}\phi(\zeta)^{y_{2}-y_{1}-1}\phi'(\zeta) A_{1}(\phi(\zeta))^{\epsilon_{2}}\e(\varphi(\zeta))\widehat{\e}(\varphi(\zeta))A_{0}(\phi(\zeta))^{\epsilon_{1}} d\zeta \nonumber \\
& +\frac{1}{2\pi i} \oint_{\gamma_{\mathbb{C}}}\oint_{\gamma_{\mathbb{C}}} \frac{\mathcal{R}_{N}^{\mathcal{W}}(\omega,\zeta)}{\widehat{h}(\omega)h(\zeta)} \frac{\widehat{\lambda}(\varphi(\omega))^{\frac{L}{2}-x_{2}}}{\phi(\omega)^{\frac{M+N}{2}-y_{2}}} \frac{\widehat{\lambda}(\varphi(\zeta))^{x_{1}}}{\phi(\zeta)^{y_{1}+1}} \phi'(\omega)\phi'(\zeta) A_{1}(\phi(\omega))^{\epsilon_{2}}\e(\varphi(\omega))  \widehat{\e}(\varphi(\zeta)) A_{0}(\phi(\zeta))^{\epsilon_{1}} d\zeta d\omega, \label{simplified formula K 2x2}
\end{align}
where the various quantities that appear in the integrands depend on
\begin{align*}
& a_{\pm} = a_{1,1}a_{0,0} \pm a_{0,1}a_{1,0}, & & b_{\pm} = b_{0,1}b_{1,1}\pm b_{0,0}b_{1,0}, \\
& c_{0} =  (a_{0,0}b_{1,1}+a_{1,0}b_{0,0})(a_{1,1}b_{0,1}+a_{0,1}b_{1,0}), & & c_{1} = (a_{0,1}b_{1,1}+a_{1,1}b_{0,0})(a_{1,0}b_{0,1}+a_{0,0}b_{1,0}), \\
& d=a_{0,0}b_{1,1}+a_{1,0}b_{0,0},
\end{align*}
and are given as follows.
\begin{itemize}
\item[(a)] If $a_{-} = 0$, then 
\begin{align*}
& \phi(\zeta) = \tfrac{\zeta^{2}-b_{-}^{2}}{2(c_{0}+c_{1})}, & & h(\zeta) = 1, & & \widehat{h}(\zeta) = \zeta, \\
& \e(\varphi(\zeta))^{T} = \begin{pmatrix}
1 & \frac{\zeta + b_{-}}{2d}
\end{pmatrix}, & & \widehat{\e}(\varphi(\zeta)) = \begin{pmatrix}
\frac{\zeta-b_{-}}{2\zeta} & \frac{d}{\zeta}
\end{pmatrix}, & & \widehat{\lambda}(\varphi(\zeta)) = \tfrac{a_{+}(\zeta-\zeta_{1})(\zeta - \zeta_{2})}{4(c_{0}+c_{1})}, \\
& \zeta_{1} = - b_{+} - \tfrac{2a_{1,0}b_{0,0}b_{0,1}}{a_{0,0}}, & & \zeta_{2} = -b_{+}- \tfrac{2 a_{0,0}b_{1,0}b_{1,1}}{a_{1,0}}, & & \mathcal{W}(\omega) = \tfrac{\widehat{\lambda}(\varphi(\omega))^{L}\phi'(\omega)}{\phi(\omega)^{\frac{M+N}{2}}h(\omega)\widehat{h}(\omega)}.
\end{align*}
The contour $\gamma_{\mathbb{C}}$ is a circle oriented positively and surrounding both $b_{-}$ and $-b_{-}$. 
\item[(b)] If $a_{-} \neq 0$, then we define
\begin{align}\label{def of zpm and c}
& z_{\pm} = \frac{1}{a_{-}^{2}}\Big( -(c_{0}+c_{1})\pm \sqrt{(c_{0}+c_{1})^{2}-a_{-}^{2}b_{-}^{2}} \Big), & & c = \frac{\sqrt{|z_{-}|}-\sqrt{|z_{+}|}}{\sqrt{|z_{-}|}+\sqrt{|z_{+}|}}.
\end{align}
They satisfy $z_{-}<z_{+}<0$, $c \in (0,1)$, and we have
\begin{align}
& \phi(\zeta) = \frac{z_{+}-z_{-}}{4\zeta}(\zeta-c)(\zeta -c^{-1}), & & h(\zeta) = \zeta^{N}, \quad \widehat{h}(\zeta) = \zeta^{N-2}(\zeta-1)(\zeta + 1), \nonumber \\
& \e(\varphi(\zeta))^{T} = \begin{pmatrix}
1 & \frac{b_{-}-a_{-}\phi(\zeta)+\eta(\zeta)}{2d} 
\end{pmatrix}, & & \widehat{\e}(\varphi(\zeta)) = \begin{pmatrix}
\frac{a_{-}\phi(\zeta)+\eta(\zeta)-b_{-}}{2\eta(\zeta)} & \frac{d}{\eta(\zeta)}
\end{pmatrix}, \label{e and einv for 2x2 a- neq 0} \\
& \eta(\zeta) = a_{-}\frac{z_{+}-z_{-}}{4}\left( \zeta-\zeta^{-1} \right), & & \widehat{\lambda}(\varphi(\zeta)) = \tfrac{a_{+} \phi(\zeta) + b_{+} + \eta(\zeta)}{2}, \quad \mathcal{W}(\omega) = \frac{\widehat{\lambda}(\varphi(\omega))^{L}\phi'(\omega)}{\phi(\omega)^{\frac{M+N}{2}}h(\omega)\widehat{h}(\omega)}. \nonumber
\end{align}
The contour $\gamma_{\mathbb{C}}$ is a closed curve surrounding both $c$ and $c^{-1}$ in the positive direction, but not surrounding $0$.
\end{itemize}
\end{theorem}
\begin{remark}
If $a_{-} \neq 0$, the function $\widehat{\e}$ has simple poles at $1$ and $-1$. However, because $\phi'(\zeta)$ has two simple zeros at $1$ and $-1$, the only poles of the integrand in \eqref{simplified formula K 2x2} are $0$, $c$ and $c^{-1}$. 
\end{remark}
\begin{proof}
Define
\begin{align*}
\Delta = \Delta(z) = a_{-}^{2}z^{2} + 2(c_{0}+c_{1})z+ b_{-}^{2} = \begin{cases}
a_{-}^{2}(z-z_{+})(z-z_{-}), & \mbox{if } a_{-} \neq 0, \\
2(c_{0}+c_{1})(z-z_{1}), & \mbox{if } a_{-} = 0,
\end{cases}
\end{align*}
where $z_{-},z_{+}$ are given by \eqref{def of zpm and c} and $z_{1}=- \frac{b_{-}^{2}}{2(c_{0}+c_{1})}<0$. The eigenvalues of $A(z) = A_{0}(z) A_{1}(z)$ are given by
\begin{align*}
& \widehat{\lambda}_{1}(z) = \frac{1}{2} \Big(a_{+} z + b_{+} +\sqrt{\Delta(z)}\Big), \qquad \widehat{\lambda}_{2}(z) = \frac{1}{2} \Big(a_{+} z + b_{+} - \sqrt{\Delta(z)}\Big),
\end{align*}
where the branch for $\sqrt{\Delta(z)}$ is taken as follows:
\begin{align*}
& \mbox{if $a_{-} \neq 0$, $\sqrt{\Delta(z)}$ is analytic in $\mathbb{C}\setminus [z_{-},z_{+}]$ and $\sqrt{\Delta(z)} \sim a_{-}z$ as $z \to \infty$,} \\
& \mbox{if  $a_{-} = 0$, $\sqrt{\Delta(z)}$ is analytic in $\mathbb{C}\setminus (-\infty,z_{1}]$ and $\sqrt{\Delta(z)} >0$ for $z > z_{1}$.}
\end{align*}
A simple computation shows that $A(z)$ can be diagonalized as in \eqref{diagonal form of A} with $\widehat{\Lambda}(z) = \diag(\widehat{\lambda}_{1}(z), \widehat{\lambda}_{2}(z))$ and 
\begin{align}\label{E and Einv for 2x2}
E(z) = \begin{pmatrix}
1 & 1 \\
\frac{b_{-}-a_{-}z+\sqrt{\Delta}}{2d} & \frac{b_{-}-a_{-}z-\sqrt{\Delta}}{2d}
\end{pmatrix}, \qquad E(z)^{-1} = \begin{pmatrix}
\frac{a_{-}z+\sqrt{\Delta}-b_{-}}{2\sqrt{\Delta}} & \frac{d}{\sqrt{\Delta}} \\
\frac{a_{-}z-\sqrt{\Delta}-b_{-}}{-2\sqrt{\Delta}} & \frac{d}{-\sqrt{\Delta}}
\end{pmatrix}.
\end{align}
In particular, $A$ satisfies Assumption \ref{ass: matrix A}. Let $\mathcal{M}$ be the genus $0$ Riemann surface associated to $\{(z,\eta)\in \mathbb{C}^{2} : \eta^{2} = \Delta(z) \}$. We choose the numbering of the sheets such that $\eta = \sqrt{\Delta(z)}$ on the first sheet, and $\eta = -\sqrt{\Delta(z)}$ on the second sheet. If $a_{-} = 0$, we note that the maps 
\begin{align*}
\varphi(\zeta) = \bigg(\frac{\zeta^{2}-b_{-}^{2}}{2(c_{0}+c_{1})}, \zeta\bigg) \qquad \mbox{ and } \qquad \varphi^{-1}((z,\eta)) = \eta
\end{align*}
are bijections from $\widehat{\mathbb{C}}$ to $\mathcal{M}$ and from $\mathcal{M}$ to $\widehat{\mathbb{C}}$, respectively, and therefore the claim is obtained in a similar way as in the proof of Theorem \ref{thm: 2x1} (we omit further details). We now consider the case $a_{-} \neq 0$. The function
\begin{align*}
\varphi^{-1}((z,\eta)) = \frac{2z+2a_{-}^{-1}\eta-(z_{+}+z_{-})}{z_{+}-z_{-}}
\end{align*}
maps the upper sheet to $\{\zeta:|\zeta|>1\}$ and the lower sheet to $\{\zeta:|\zeta|<1\}$. The inverse map is given by $\varphi(\zeta) = ( \phi(\zeta), \eta(\zeta) )$ with
\begin{align}\label{z and eta for 2x2 a- neq 0}
\phi(\zeta) = \frac{z_{+}+z_{-}}{2} + \frac{z_{+}-z_{-}}{4}(\zeta + \zeta^{-1}) \quad \mbox{ and } \quad \eta(\zeta) = a_{-}\frac{z_{+}-z_{-}}{4}\left( \zeta-\zeta^{-1} \right).
\end{align}
The function $\varphi$ satisfies
\begin{align*}
& \varphi(1) = (z_{+},0), \qquad \varphi(-1) = (z_{-},0), \qquad \varphi(\infty) = \infty^{(1)}, \qquad \varphi(0) = \infty^{(2)}, \\
& \varphi^{-1}(0^{(2)}) = \varphi^{-1}((0,-a_{-}\sqrt{z_{+}z_{-}})) = c, \qquad \varphi^{-1}(0^{(1)}) = \varphi^{-1}((0,a_{-}\sqrt{z_{+}z_{-}})) = c^{-1},
\end{align*}
where $c$ is defined in \eqref{def of zpm and c}. Also, by definition, $\phi(\zeta)$ vanishes at $\varphi^{-1}(0^{(1)})$ and $\varphi^{-1}(0^{(2)})$, and it has simple poles at $\zeta = 0$ and $\zeta = \infty$. Hence, it can be rewritten as 
\begin{align*}
\phi(\zeta) = \frac{z_{+}-z_{-}}{4\zeta}(\zeta-c)(\zeta -c^{-1}).
\end{align*}
The expressions \eqref{e and einv for 2x2 a- neq 0} for $\e$ and $\widehat{\e}$ follow directly from \eqref{def of efrak RS}-\eqref{def of efrak inv RS}, \eqref{E and Einv for 2x2} and \eqref{z and eta for 2x2 a- neq 0}. Since $\e$ has no zero and a simple pole at $\infty^{(2)}$, by \eqref{def of h} and Definition \ref{def: nz and nz hat}, we have 
\begin{align*}
\mathcal{Z} = \emptyset, \qquad \mathcal{Q} = \{\infty^{(2)}\}, \qquad n_{\infty^{(2)}} = -1, \qquad h(\zeta) = (\zeta-\varphi^{-1}(\infty^{(2)}))^{N-1}(\zeta-\varphi^{-1}(\infty^{(2)})) = \zeta^{N}.
\end{align*}
Similarly, since $\widehat{\e}$ has a simple zero at $\infty^{(2)}$ and simple poles at $z_{+}$ and $z_{-}$, we have
\begin{align*}
\widehat{\mathcal{Z}} = \{\infty^{(2)}\}, \qquad \widehat{\mathcal{Q}} = \{z_{+}, z_{-}\}, \qquad \widehat{n}_{\infty^{(2)}} = 1, \qquad \widehat{n}_{z_{+}} = \widehat{n}_{z_{-}} = -1
\end{align*}
and therefore, using \eqref{def of h hat} we obtain
\begin{align*}
\widehat{h}(\zeta) = (\zeta-\varphi^{-1}(\infty^{(2)}))^{N-1}\frac{(\zeta-\varphi^{-1}(z_{+}))(\zeta-\varphi^{-1}(z_{-}))}{\zeta-\varphi^{-1}(\infty^{(2)})} = \zeta^{N-2}(\zeta-1)(\zeta + 1).
\end{align*}
Since $-\sum_{\z \in \mathcal{Z}\cup \mathcal{Q}} n_{\z} = 1$, Theorem \ref{thm: reprod Rcal U} (e) implies that $\mathfrak{R}_{N}^{\mathcal{W}} = \mathcal{R}_{N}^{\mathcal{W}}$. Finally, since $\gamma$ can be deformed into any closed curve surrounding $0$, and since $\varphi^{-1}(0^{(1)})=c^{-1}$ and $\varphi^{-1}(0^{(2)})=c$, it follows that $\gamma_{\mathbb{C}}$ can be chosen as the union of two small circles; one surrounds $c$ and the other one surrounds $c^{-1}$, but none of them surround $0$. The formula \eqref{simplified formula K 2x2} now follows from a direct application of Theorem \ref{thm: double contour integrals simplifications}.
\end{proof}
\section{Proofs of Theorems \ref{thm: reproducing kernel Riemann surface} and \ref{thm: reprod Rcal U}}\label{section: main results}
\subsection{Proof of Theorem \ref{thm: reproducing kernel Riemann surface}}\label{proof kernel M}
The existence of $R_{N}^{\lambda}$ follows directly from \eqref{def of RM} and the assumption that $\mathcal{R}_{N}^{W}$ exists. Given $z \in \mathbb{C}$, recall that $z^{(k)}$ denotes the point on the $k$-th sheet of $\mathcal{M}$ whose projection on $\mathbb{C}$ is $z$. Because $E(z)$ is invertible for all but finitely many $z \in \mathbb{C}$, for any $P \in \mathcal{P}_{N-1}^{1 \times r}$ we have
\begin{align*}
P(z) \e(\z) \equiv 0 \; \Leftrightarrow \; P(z) \e(z^{(j)}) \equiv 0 \quad \forall j \in \{1,\ldots,r\} \; \Leftrightarrow \; P(z)E(z) \equiv 0 \; \Leftrightarrow \; P(z) \equiv 0
\end{align*}
from which we conclude that $\dim L_{N} =rN$. Similarly, for any $P \in \mathcal{P}_{N-1}^{r \times 1}$, we have
\begin{align*}
\widehat{\e}(\z)P(z)  \equiv 0 \; \Leftrightarrow \; \widehat{\e}(z^{(j)})P(z)  \equiv 0 \quad \forall j \in \{1,\ldots,r\} \; \Leftrightarrow \; E(z)^{-1}P(z) \equiv 0 \; \Leftrightarrow \; P(z) \equiv 0,
\end{align*}
and therefore $\dim \widehat{L}_{N} =rN$. Since $\mathcal{R}_{N}^{W}(w,z)$ is a bivariate $r \times r$ matrix polynomial of degree $\leq N-1$ in both $w$ and $z$, the statements (a) and (b) follow directly from \eqref{def of RM} and \eqref{def of LN and LN star}. Next, we start from \eqref{reproducing kernel}, and use \eqref{eigenvalue eigenvector decomposition of A} and \eqref{def of lambda RS}--\eqref{def of efrak inv RS} to note the following equivalences 
\begin{align*}
& \int_{\gamma} P(w) W(w)\mathcal{R}_{N}^{W}(w,z)dw = P(z), & & \forall P \in \mathcal{P}_{N-1}^{r \times r}, \; z \in \mathbb{C}, \\
\Leftrightarrow \; &  \int_{\gamma} P(w) W(w)\mathcal{R}_{N}^{W}(w,z)dw = P(z), & & \forall P \in \mathcal{P}_{N-1}^{1 \times r}, \; z \in \mathbb{C}, \\
\Leftrightarrow \; &  \int_{\gamma} P(w) E(w) \bigg(\sum_{j=1}^{r}\lambda(w^{(j)})e_{j}e_{j}^{T}\bigg)E^{-1}(w)\mathcal{R}_{N}^{W}(w,z)dw = P(z), & & \forall P \in \mathcal{P}_{N-1}^{1 \times r}, \; z \in \mathbb{C}, \\
\Leftrightarrow \; &  \int_{\gamma_{\mathcal{M}}} P(w) \e(\w) \lambda(\w)\widehat{\e}(\w)\mathcal{R}_{N}^{W}(w,z)dw = P(z), & & \forall P \in \mathcal{P}_{N-1}^{1 \times r}, \; z \in \mathbb{C}, \\
\Leftrightarrow \; &  \int_{\gamma_{\mathcal{M}}} P(w) \e(\w) \lambda(\w)R_{N}^{\lambda}(\w,\z)dw = P(z)\e(\z), & & \forall P \in \mathcal{P}_{N-1}^{1 \times r}, \; \z \in \mathcal{M}_{*}\setminus \mathcal{Q},
\end{align*}
where in the last two equations, $w$ denotes the projection of $\w$ on the complex plane, and similarly for $z$ and $\z$ in the last equation. By definition \eqref{def of LN and LN star} of $L_{N}$, this last property is equivalent to \eqref{reproducing property M}, which proves (c). The proof of (d) is similar, and we omit it. 

\subsection{Proof of Theorem \ref{thm: reprod Rcal U}}\label{subsection: proof thm}
Existence of $\mathfrak{R}_{rN}^{\mathcal{W}}$ is obvious from \eqref{def of RM}, \eqref{def of RU} and the assumption that $\mathcal{R}_{N}^{W}$ exists. The identities $\dim \mathcal{V} = \dim \widehat{\mathcal{V}} = rN$ have already been proved in Remark \ref{remark: V and V* are polynomial spaces}. The statements (a) and (b) follow directly from the definition \eqref{V and V*} and Theorem \ref{thm: reproducing kernel Riemann surface} (a)-(b). We now turn to the proof of (c). Since the sets $\mathcal{Q}$ and $\mathcal{Z}$ of Definition \ref{def: nz and nz hat} are finite, we have the following equivalences
\begin{align*}
& \int_{\gamma_{\mathcal{M}}} f(\w) \lambda(\w)R_{N}^{\lambda}(\w,\z)dw = f(\z), & & \forall f \in L_{N}, \; \z \in \mathcal{M}_{*}\setminus \mathcal{Q}, \\
\Leftrightarrow \;  & \int_{\gamma_{\mathcal{M}}} f(\w) \lambda(\w)R_{N}^{\lambda}(\w,\varphi(\zeta))dw = f(\varphi(\zeta)), & & \forall f \in L_{N}, \; \zeta \in \mathbb{C}\setminus \varphi^{-1}(\mathcal{Q}), \\
\Leftrightarrow \;  & \int_{\gamma_{\mathcal{M}}} f(\w) \lambda(\w)R_{N}^{\lambda}(\w,\varphi(\zeta))h(\zeta)dw = f(\varphi(\zeta)) h(\zeta), & & \forall f \in L_{N}, \; \zeta \in \mathbb{C}\setminus \varphi^{-1}(\mathcal{Q} \cup \mathcal{Z}), \\
\Leftrightarrow \;  &  \int_{\gamma_{\mathbb{C}}} p(\omega) \frac{\lambda(\varphi(\omega))}{h(\omega)\widehat{h}(\omega)}\widehat{h}(\omega)R_{N}^{\lambda}(\varphi(\omega),\varphi(\zeta))h(\zeta) \phi'(\omega)d\omega = p(\zeta), & & \forall p \in \mathcal{V}, \; \zeta \in \mathbb{C}, \\
\Leftrightarrow \;  & \int_{\gamma_{\mathbb{C}}} p(\omega) \mathcal{W}(\omega)  \mathfrak{R}_{rN}^{\mathcal{W}}(\omega,\zeta) d\omega = p(\zeta), & & \forall p \in \mathcal{V}, \; \zeta \in \mathbb{C},
\end{align*}
which prove (c). The statement (d) follows in a similar way, and we omit the proof. Finally, recall that only one property among \eqref{reproducing kernel}-\eqref{reproducing kernel 2} is sufficient to uniquely determine a CD kernel. Hence, Remark \ref{remark: V and V* are polynomial spaces} implies that
\begin{align*}
& -\sum_{\z \in \mathcal{Z}\cup \mathcal{Q}} n_{\z} = r-1  \; \Leftrightarrow \; \mathcal{V} = \mathcal{P}_{rN-1} \; \Leftrightarrow \; \mathfrak{R}_{rN}^{\mathcal{W}} = \mathcal{R}_{rN}^{\mathcal{W}}, \\
& -\sum_{\z \in \widehat{\mathcal{Z}}\cup \widehat{\mathcal{Q}}}\widehat{n}_{\z} = r-1 \; \Leftrightarrow \; \widehat{\mathcal{V}} = \mathcal{P}_{rN-1} \; \Leftrightarrow \; \mathfrak{R}_{rN}^{\mathcal{W}} = \mathcal{R}_{rN}^{\mathcal{W}},
\end{align*}
which, taken together, are equivalent to (e). 

\appendix

\section{On the eigenvalues and eigenvectors of $W$}\label{section: appendix eigenvector}
In this appendix we discuss some analytical properties of the eigenvalues and eigenvectors of $W$. Since $W$ is rational, there exists a scalar polynomial $p$ such that $W=p^{-1}T$, where $T$ is a polynomial matrix. For each $z$ that is not a pole of $W$, the eigenvalues $\lambda_{1}(z),\ldots,\lambda_{r}(z)$ of $W(z)$ and the eigenvalues $\theta_{1}(z),\ldots,\theta_{r}(z)$ of $T(z)$ are straightforwardly related by $\lambda_{j}(z) = p(z)^{-1}\theta_{j}(z)$, $j=1,\ldots,m$, and each eigenvector $v(z)$ of $W(z)$ satisfying $W(z)v(z)=\lambda_{k}(z)v(z)$ also satisfies $T(z)v(z)=\theta_{k}(z)v(z)$. Therefore, we restrict from now our discussion on the eigenvalues and eigenvectors of $T$. Most of the facts listed below are rather direct consequences of the classical book of \cite{Kato}.

\paragraph{On the eigenvalues of $T$.} Let $\mathcal{M}$ be the Riemann surface constructed from the zero set \eqref{Riemann surface M general case}. The eigenvalues of $T$ are (branches of) meromorphic functions on $\widehat{\mathbb{C}}:=\mathbb{C}\cup\{\infty\}$, and together they define a meromorphic function on $\mathcal{M}$. Let $\uptheta_{1}(z),\ldots,\uptheta_{s}(z)$ denote the \textit{distinct} eigenvalues of $T(z)$, and let $m_{1},\ldots,m_{s}$ be their multiplicities, with $m_{1}+\ldots + m_{s}=r$. Since $\mathcal{M}$ is compact, $s,m_{1},\ldots,m_{s}$ are constant for all $z \in \mathbb{C}\setminus \mathcal{E}$, where $\mathcal{E}$ consists of at most finitely many exceptional points.\footnote{Without the compactness of $\mathcal{M}$, we would only have that $\mathcal{E}$ is locally finite. Also, the terminology \textit{exceptional points} to denote points of $\mathcal{E}$ is standard \cite{Kato}.} Any point that is a branch point for some of the $\uptheta_{j}$'s belongs to $\mathcal{E}$; however $\mathcal{E}$ may also contain other points, see \cite[page 64, Example 1.1]{Kato}. Also, because $T$ is analytic in $\mathbb{C}$, the functions $z \mapsto \uptheta_{j}(z)$ are \textit{continuous} at any point $z \in \mathbb{C}$, also at a branch point. This fact is essentially a consequence of Rouch\'{e}'s theorem, see \cite[p. 122]{Knopp}. In particular the $\uptheta_{j}$'s have no pole in $\mathbb{C}$, although they can have a pole at $\infty$. 

\paragraph{On the eigenvectors of $T$.} Eigenprojections are standard tools in analytic perturbation theory. They are closely related to the eigenvectors (see below), but allow for a simplified analysis and their properties have been studied in great depth in \cite{Kato}. The eigenprojection $\mathsf{P}_{k}(z)$ associated to $\uptheta_{k}(z)$ is defined for $z \in \mathbb{C}\setminus \mathcal{E}$ by
\begin{align*}
\mathsf{P}_{k}(z) = -\frac{1}{2\pi i}\oint_{\Gamma_{k,z}} (T(z)-\uplambda)^{-1}d\uplambda, \qquad k=1,\ldots,s,
\end{align*}
where $\Gamma_{k,z}$ is a small contour in the complex plane which surrounds $\uptheta_{k}(z)$, but does not surround $\uptheta_{j}(z)$, $j \neq k$. Given $z \in \widehat{\mathbb{C}}$, we let $z^{(k)}$ denotes the point on the $k$-th sheet of $\mathcal{M}$ whose projection on $\widehat{\mathbb{C}}$ is $z$. An individual $\mathsf{P}_{k}$ has the same branch cut as $\uptheta_{k}$, and taken together the $\mathsf{P}_{k}$'s naturally define a meromorphic function on $\mathcal{M}$, which can only have poles at 
\begin{align*}
\{\infty^{(1)},\ldots,\infty^{(r)}\} \cup \bigcup_{j=1}^{r}\mathcal{E}^{(j)},
\end{align*}
see \cite[Chapter II, Sections 4--6]{Kato}. As its name suggests, the operator $\mathsf{P}_{k}(z)$ is a projection, and it satisfies \cite[pages 40]{Kato}
\begin{align*}
(T(z)-\uptheta_{k}(z))^{m_{k}}\mathsf{P}_{k}(z) = 0_{r}.
\end{align*}
Furthermore, the images of $\mathsf{P}_{k}(z)$ and of $\mathsf{P}_{k}(z')$ are isomorphic for any $z,z' \in \mathbb{C}\setminus \mathcal{E}$, and $\dim \im \mathsf{P}_{k}(z) = m_{k}$ for $z \in \mathbb{C}\setminus \mathcal{E}$, see \cite[page 68]{Kato}. 

\medskip By Assumption \ref{ass: weight diag}, $T(z)$ is diagonalizable for all $z \in \mathbb{C}\setminus \mathcal{D}$ where $\mathcal{D}$ is a finite set, and we assume without loss of generality that $\mathcal{E} \subset \mathcal{D}$. This implies in particular that the algebraic eigenspaces $\im \mathsf{P}_{k}(z)$, $k=1,\ldots,s$ coincide with the geometric eigenspaces, i.e. we have
\begin{align*}
(T(z)-\uptheta_{k}(z))\mathsf{P}_{k}(z) = 0_{r}, \qquad \mbox{or equivalently} \qquad T(z)\mathsf{P}_{k}(z) = \uptheta_{k}(z)\mathsf{P}_{k}(z),
\end{align*}
for all $k=1,\ldots,s$. Fix $z \in \mathbb{C}\setminus \mathcal{D}$. For each $k=1,\ldots,s$, we take $m_{k}$ linearly independent columns of $\mathsf{P}_{k}(z)$. Together these columns form a matrix of eigenvectors $E(z)$, and we choose the numbering of the columns such that
\begin{align*}
T(z)E(z) = E(z)\diag (\theta_{1}(z),\ldots,\theta_{r}(z)).
\end{align*} 
Since the columns of $E$ are (branches of) meromorphic functions, they remain linearly independent for all but a finite number of points. Redefining $\mathcal{D}$ is necessary, we can assume that $E(z)$ is invertible for all values of $\mathbb{C}\setminus \mathcal{D}$. This finishes the construction of a matrix of eigenvectors $E$ whose columns define a meromorphic function $\e$ on $\mathcal{M}$ as in \eqref{def of efrak RS}.

\medskip It directly follows from the above properties of $E$ and from Cramer's formula that the rows of $E^{-1}$ define also a meromorphic function $\widehat{\e}$ on $\mathcal{M}$ as in \eqref{def of efrak inv RS}.

\section{On the CD formula \eqref{CD simplified} for $\mathcal{R}_{N}^{W}$}\label{section: appendix}
The goal of this appendix is to rewrite \eqref{def of mathcal R} in the form \eqref{CD simplified}. Assume that $P_{N}^{\mathrm{L}}$ exists and is unique. Then the solution to the RH problem for $Y$ exists, is also unique, and can be explicitly written in term of MOPs as follows \cite[eq (4.31)]{DK}:
\begin{equation}\label{Y definition}
Y(z) = \begin{pmatrix}
P_{N}^{\mathrm{L}}(z) & \displaystyle \frac{1}{2\pi i} \int_{\gamma} P_{N}^{\mathrm{L}}(s) W(s) \frac{ds}{s-z} \\
-2\pi i Q_{N-1}^{\mathrm{L}}(z) & \displaystyle - \int_{\gamma} Q_{N-1}^{\mathrm{L}}(s) W(s) \frac{ds}{s-z}
\end{pmatrix}, \qquad z \in \mathbb{C}\setminus \gamma.
\end{equation} 
Since $Y$ satisfies $\det Y \equiv 1$, the existence of $Y^{-1}$ follows from that of $Y$. An explicit expression for $Y^{-1}$ is not clear from \eqref{Y definition}, but can be easily obtained by considering the RH problem for $Y^{-1}$, which is as follows. 
\vspace{-0.3cm}
\subsubsection*{RH problem for $Y^{-1}$}
\begin{itemize}
\item[(a)] $Y^{-1} : \mathbb{C}\setminus \gamma \to \mathbb{C}^{2r \times 2r}$ is analytic.
\item[(b)] The limits of $Y^{-1}(z)$ as $z$ approaches $\gamma_{0}$ from left and right exist, are continuous on $\gamma_{0}$, and are denoted by $Y_+^{-1}$ and $Y_-^{-1}$, respectively. Furthermore, they are related by
\begin{equation*}
Y_{+}^{-1}(z) = \begin{pmatrix}
I_{r} & -W(z) \\ 0_{r} & I_{r}
\end{pmatrix}Y_{-}^{-1}(z), \hspace{0.5cm} \mbox{ for } z \in \gamma_{0}.
\end{equation*}
\item[(c)] As $z \to \infty$, we have $Y^{-1}(z) = \begin{pmatrix}
z^{-N}I_{r} & 0_{r} \\ 0_{r} & z^{N}I_{r}
\end{pmatrix} \left(I_{2r} + \bigO(z^{-1})\right) $.  \\[0.2cm]
As $z \to z_{\star} \in \gamma\setminus \gamma_{0}$, we have $Y^{-1}(z) = \bigO(\log (z-z_{\star}))$.
\end{itemize}
It is easily verified from \eqref{orthogonality for PNR and PNL}--\eqref{orthogonality for Q} that the unique solution to the above RH problem is given by
\begin{equation}\label{Y inv definition}
Y^{-1}(z) = \begin{pmatrix}
\displaystyle - \int_{\gamma}  W(s) Q_{N-1}^{\mathrm{R}}(s) \frac{ds}{s-z} & \displaystyle -\frac{1}{2\pi i} \int_{\gamma}  W(s)P_{N}^{\mathrm{R}}(s) \frac{ds}{s-z} \\
2\pi i Q_{N-1}^{\mathrm{R}}(z) & P_{N}^{\mathrm{R}}(z)
\end{pmatrix}, \qquad z \in \mathbb{C}\setminus \gamma.
\end{equation}
The CD formula \eqref{CD simplified} is now simply obtained by substituting \eqref{Y definition} and \eqref{Y inv definition} in \eqref{def of mathcal R}. 

\paragraph{Acknowledgment.} I am very grateful to Arno Kuijlaars for our many interesting discussions on the topic of this article, and for helpful comments on some early drafts. This work is supported by the European Research Council, Grant Agreement No. 682537.

\end{document}